\documentclass[11pt,reqno]{amsart}

\usepackage[a4paper,left=35mm,right=35mm,top=30mm,bottom=30mm,marginpar=25mm]{geometry}

\usepackage{amsmath}
\usepackage{amssymb}
\usepackage{amsthm}
\usepackage[latin1]{inputenc}
\usepackage{eurosym}
\usepackage[dvips]{graphics}
\usepackage{graphicx}
\usepackage{epsfig}
\usepackage{hyperref}
\usepackage{dsfont}
\usepackage{color}

\usepackage[displaymath,mathlines]{lineno}

\allowdisplaybreaks

\usepackage{hyperref}

\usepackage{ifthen}

\makeindex

\newcommand{\argmin}{\operatorname{argmin}}
\newcommand{\argmax}{\operatorname{argmax}}

\newcommand{\Rr}{{\mathbb{R}}}

\newcommand{\Tt}{{\mathbb{T}}}

\newcommand{\Hh}{{\overline{H}}}

\newcommand{\Aa}{{\mathcal{A}}}

\newcommand{\epsi}{\varepsilon}

\def\leq{\leqslant}
\def\geq{\geqslant}

\numberwithin{equation}{section}

\newtheoremstyle{thmlemcorr}{10pt}{10pt}{\itshape}{}{\bfseries}{.}{10pt}{{\thmname{#1}\thmnumber{
                        #2}\thmnote{ (#3)}}}
\newtheoremstyle{thmlemcorr*}{10pt}{10pt}{\itshape}{}{\bfseries}{.}\newline{{\thmname{#1}\thmnumber{
                        #2}\thmnote{ (#3)}}}
\newtheoremstyle{defi}{10pt}{10pt}{\itshape}{}{\bfseries}{.}{10pt}{{\thmname{#1}\thmnumber{
                        #2}\thmnote{ (#3)}}}
\newtheoremstyle{remexample}{10pt}{10pt}{}{}{\bfseries}{.}{10pt}{{\thmname{#1}\thmnumber{
                        #2}\thmnote{ (#3)}}}
\newtheoremstyle{ass}{10pt}{10pt}{}{}{\bfseries}{.}{10pt}{{\thmname{#1}\thmnumber{
                        A#2}\thmnote{ (#3)}}}

\theoremstyle{thmlemcorr}
\newtheorem{theorem}{Theorem}
\numberwithin{theorem}{section}
\newtheorem{lemma}[theorem]{Lemma}

\newtheorem{proposition}[theorem]{Proposition}

\theoremstyle{thmlemcorr*}
\newtheorem{theorem*}{Theorem}
\newtheorem{lemma*}[theorem]{Lemma}
\newtheorem{corollary*}[theorem]{Corollary}
\newtheorem{proposition*}[theorem]{Proposition}
\newtheorem{problem*}[theorem]{Problem}
\newtheorem{conjecture*}[theorem]{Conjecture}

\theoremstyle{defi}

\theoremstyle{remexample}
\newtheorem{remark}[theorem]{Remark}

\newtheorem{pro}[theorem]{Proposition}

\theoremstyle{ass}

\begin{document}

        \title[One-dimensional, first-order, stationary
mean-field games]{One-dimensional stationary
        mean-field games with local coupling}
        
        \author{Diogo A. Gomes}
        \address[D. A. Gomes]{
                King Abdullah University of Science and Technology (KAUST), CEMSE Division and  
                KAUST SRI, Uncertainty Quantification Center in Computational Science and Engineering, Thuwal 23955-6900, Saudi Arabia.}
        \email{diogo.gomes@kaust.edu.sa}
        \author{Levon Nurbekyan}
        \address[L. Nurbekyan]{
                King Abdullah University of Science and Technology (KAUST), CEMSE Division, Thuwal 23955-6900, Saudi Arabia. }
        \email{levon.nurbekyan@kaust.edu.sa}
        \author{Mariana Prazeres}
        \address[M. Prazeres]{
                King Abdullah University of Science and Technology (KAUST), CEMSE Division, Thuwal 23955-6900, Saudi Arabia. }
        \email{mariana.prazeres.1@kaust.edu.sa}

        
        \thanks{
                D. Gomes, L. Nurbekyan and M. Prazeres were partially supported by KAUST baseline and start-up funds.
        }
        \date{\today}

        \begin{abstract}
                A standard assumption in mean-field game (MFG) theory is that the coupling between the Hamilton-Jacobi equation and the transport equation is monotonically non-decreasing in the density of the population. In many cases, this assumption implies the existence and uniqueness of solutions. Here, we drop that assumption and construct explicit solutions for one-dimensional MFGs. These solutions exhibit phenomena  not present in monotonically increasing MFGs:  low-regularity, non-uniqueness, and the formation of regions with no agents.
        \end{abstract}
        
        \maketitle

        \section{Introduction}
        
        Mean-field game (MFG) theory \cite{cardaliaguet, GPV,   GS, llg2} describes on-cooperative differential games with infinitely many identical players. These games were introduced  by Lasry and Lions
\cite{ll1, ll2, ll3}  and, independently around the same time, by
Huang, Caines and Malham\'{e} \cite{Caines2, Caines1}. Often, MFGs are given by a Hamilton-Jacobi equation coupled with a Fokker-Planck equation. A standard example is the 
        stationary, one-dimensional, first-order MFG:         \begin{equation}\label{main}
       \begin{cases}
        \frac{(u_x+p)^2}{2}+V(x)=g(m)+\Hh,\\
        -(m(u_x+p))_x =0,  
        \end{cases}
        \end{equation}
with its elliptic regularization,
   \begin{equation}\label{maine}
   \begin{cases}
   -\epsilon u_{xx}+\frac{(p+u_x)^2}{2}+V(x)=\Hh+ g(m)\\
   -\epsilon m_{xx}-((p+u_x) m)_x=0, 
   \end{cases}
   \end{equation}
where $\epsilon>0$.        
Here, $p$ is a fixed real number and the unknowns are the constant $\Hh$ and the functions $u$ and $m$.  The function $g$ is $C^\infty$ on $\Rr^+$. To simplify the presentation, we consider the periodic case and work in the one-dimensional torus, $\Tt$. Accordingly,  $V : \Tt \rightarrow \Rr$ is a $C^\infty$ potential. We search for periodic solutions, $u,m :\Tt \rightarrow \Rr$.
Here, we examine this problem and
attempt to understand its features in terms of
the monotonicity properties of  $g$.

A standard assumption in MFGs is that $g$  is increasing. Heuristically, this assumption means that agents prefer sparsely populated areas. In this case, the existence and uniqueness of smooth solutions to \eqref{main} is well understood for stationary problems \cite{GPM1, GR, GM, PV15}, weakly coupled MFG systems \cite{GPat2}, the obstacle MFG problem \cite{GPat} and extended MFGs \cite{GPatVrt}. In the time-dependent setting, similar results are obtained in \cite{GPim2, GPim1, GVrt} for standard MFGs and in \cite{GP13, GP2} for forward-forward problems. The theory of weak solutions is also well developed for first-order and second-order problems (see \cite{Cd1, Cd2,MR3358627} and \cite{cgbt, FG2, porretta, porretta2}, respectively). Congestion problems, see \cite{GMit,  GVrt2, Graber2}, are also of interest and our results extend straightforwardly  \cite{GNP2}.

The case of a non-monotonically increasing $g$ is relevant: if $g$\ is decreasing, agents prefer clustering in high-density areas.  
The case where $g$ first decreases and then increases is also natural; here, agents have a preferred density given by the minimum of $g$.
However, little is known about the properties of \eqref{main}
when $g$ is not increasing. One of the few known cases is a second-order MFG with $g(m)=-\ln m$ and a quadratic cost. In this case,  due to the particular structure of the equations, there are explicit solutions, see \cite{gueant3, llg2}.

A triplet, $(u,m,\Hh),$ solves \eqref{main} if
        \begin{itemize}
                \item[i.] $u$ is a Lipschitz viscosity solution of the first equation in \eqref{main};
                \item[ii.] $m$ is a probability density; that is,        \begin{equation*}
                m\geq 0,\quad \int\limits_{\Tt} m =1;
                \end{equation*}
                \item[iii.] $m$ is a weak (distributional) solution of the second equation in \eqref{main}.
        \end{itemize}
        Because \eqref{main} is invariant under addition of constants to $u$, we assume that $u(0)=0$.
Here, $u$\ is Lipschitz continuous. However, $m$\ can be discontinuous. In this case, viscosity solutions of the first equation in         
         \eqref{main} are interpreted as discontinuous viscosity solutions; see, for example, \cite{Barlesbook} and the discussion in  Section \ref{discsec}.

Our problem is one-dimensional and the Hamiltonian is convex. If $u$ is a piecewise $C^1$ function and $m$\ is continuous, then  $u$\ is a viscosity solution if
the following conditions hold:\begin{itemize}
\item[a.] $u$ solves the equation at the points where it is $C^1$ and $m$ is continuous;
\item[b.] $\lim\limits_{x\to x_0^-}u_x(x) \geq \lim\limits_{x\to x_0^+}u_x(x)$ at points of discontinuity of $u_x$. 
\end{itemize}

When $g$ is not increasing, \eqref{main} may not admit  $m$ continuous. Solutions must, therefore,  be considered in the framework of discontinuous viscosity solutions. In this case, the above characterization of one-dimensional viscosity solutions is not valid, and \eqref{main} admits a large family of discontinuous viscosity solutions (see Section \ref{discsec}). On the other hand,  solutions that  satisfy the above conditions (a. and b.) have nice structural properties that we discuss in this paper. Furthermore, in their analysis we see the appearance of discontinuities in $m$, which in turn motivates the study of discontinuous viscosity solutions. Overall, these conditions seem to be good selection criteria for discontinuous solutions of \eqref{main}.

We call solutions that satisfy conditions a. and b. \textit{regular} (they can still be discontinuous). In this paper, we always consider regular solutions except in Section \ref{discsec}, where we discuss general discontinuous viscosity solutions. Furthermore, when $m$ is continuous the term ``regular" is superfluous. Thus, except in Section \ref{discsec},  we refer to  regular solutions.

Our goal is to solve \eqref{main} explicitly and to understand the qualitative behavior of solutions. For that, in Section \ref{cfor},  we reformulate \eqref{main} in terms of the \textit{current,}
        \begin{equation}\label{eq: j}
        j=m(u_x+p).
        \end{equation}
From the second equation in \eqref{main}, $j$ is constant. Thus, the current becomes the main parameter in our analysis.

While  we focus our attention into  non-increasing MFGs, our methods are also valid for increasing MFGs.
To illustrate and contrast these two cases, we begin our analysis in Section \ref{sec: g_increasing} by addressing the latter. For $j>0$, we show the
existence of a unique smooth solution. However, for $j=0$,  we uncover new phenomena: the
existence of non-smooth solutions and the lack of uniqueness. 

In Section \ref{sec:monelliptic}, we consider the elliptic regularization of monotone MFGs. We establish a new variational principle that gives the existence and uniqueness of smooth solutions. Moreover, we address the vanishing viscosity problem using  $\Gamma$-convergence.

In Section \ref{sec: g_decreasing}, we study regular solutions of \eqref{main} for non-increasing $g$.  In this case, if $j\neq0$, $m>0$. However, for certain values of $j$,  \eqref{main} does not have continuous solutions. In contrast, if  $j$ is large enough, \eqref{main}  has a unique smooth solution. Moreover, if $V$ has a single point of maximum, there exists a unique solution of \eqref{main} for each $j>0$. If $V$ has multiple maxima, there are  multiple solutions. If $j=0$, the behavior of \eqref{main} is more complex and  $m$ can be discontinuous or  vanish.

Next, in Section \ref{discsec}, we consider MFGs with a decreasing nonlinearity, $g$, and discuss the properties of discontinuous viscosity solutions.  

Subsequently, in Section \ref{sec:amonotoneelliptic}, we study the elliptic regularization of anti-monotone MFGs. There, we use calculus of variations methods to prove the existence of a solution.

In Section \ref{sec: regmodes}, we examine the regularity of solutions as a function of the current and,  in Section \ref{sec: asymptotics}, we study the asymptotic behavior of solutions of \eqref{main} as  $j$ converges to $0$ and $\infty$. 
Finally, in Sections \ref{sec: j} and \ref{sec: p}, we analyze the regularity of $\Hh$ in terms of $j$ and $p$.

\section{The current formulation and regularization} 
\label{cfor}

Here, we discuss the current formulation of \eqref{main} and \eqref{maine}. After some elementary computations, we show that the current formulation of \eqref{maine} is the Euler-Lagrange equation of a suitable functional.
 
\subsection{Current formulation} 
 
Let $j$ be given by \eqref{eq: j}. From the second equation in \eqref{main},
$j$ is constant.
We split our analysis into the cases, $j\neq 0$ and $j=0$.

If $j\neq 0$,  $m(x) \neq 0$ for all $x \in \Tt$ and $u_x+p=j/m$.
        Thus, \eqref{main} can be written as
        \begin{equation}\label{eq: currentform}
        \begin{cases}
       F_j(m)=\Hh-V(x),\\
        m>0,\ \int\limits_{\Tt} m dx=1,\\
        \int\limits_{\Tt} \frac{1}{m} dx=\frac{p}{j}, 
        \end{cases}
        \end{equation}
 where $F_j(m)= \frac{j^2}{2m^2}-g(m)$. 
For each $x$, the first equation in \eqref{eq: currentform} is an algebraic equation for $m$. If $g$ is increasing, for each $x \in \Tt$ and $\Hh \in \Rr,$ there exists a unique solution. In contrast, if $g$ is not increasing, there may exist multiple solutions, as we discuss later. 
        
For $j=0,$  \eqref{main} gives
        \begin{equation}\label{eq: currentform_j=0_intro}
        \begin{cases}
        \frac{(u_x+p)^2}{2}-g(m)=\Hh-V(x),\\
        m \geq 0,\ \int\limits_{\Tt} m dx=1,\\
        m(u_x+p)=0.
        \end{cases}
        \end{equation}        
       From the last equation in \eqref{eq: currentform_j=0_intro}, either $m=0,$ in which case $u$ solves
       \[\frac{(u_x+p)^2}{2}-g(0)=\Hh-V(x),
       \]
       or $m>0$ and $g(m)+\Hh-V(x)=0$. Hence, if $g$ is increasing or decreasing, $m(x)$ is determined in a unique way; otherwise, multiple solutions can occur. 

\subsection{Elliptic regularization}

  Now, we consider the elliptic MFG \eqref{maine}.
From the second equation in that system, we conclude that 
 \[
 j=\epsilon m_x+m (p+u_x)
 \]  
 is constant.
 Thus, we solve for $u_x$\ and replace it in the first equation. Accordingly, we get
 \begin{equation}
 \label{monly}
 -\epsilon \left(\frac{j-\epsilon m_x}{m}\right)_{x}+\frac{(j-\epsilon m_x)^2}{2
m^2}+V(x)=\Hh+ g(m).
 \end{equation}
Then, using the identity
\[
\epsilon \frac{(j-\epsilon m_x) m_x}{m^2}+\frac{(j-\epsilon
m_x)^2}{2
m^2}=\frac{j^2-\epsilon^2 m_x^2}{2m^2}, 
\]         
we obtain the following equation for $m$:
\begin{equation}
\label{efm}
\epsilon^2 \frac{m_{xx}}{m}  - \epsilon^2 \frac{m_x^2}{2 m^2}+F_j(m)=\Hh-V(x).
\end{equation}
Now, let $\Phi_j$ be such that $\Phi_j'(m)=F_j(m)$; that is, 
\[
\Phi_j(m)=-\frac{j^2}{2 m}-G(m), 
\]
where $G'(m)=g(m)$. 
Then, \eqref{efm} is the Euler-Lagrange equation of the functional
\begin{equation}
\label{mainef}
\int_{\Tt} \epsilon^2 \frac{m_x^2}{2 m}-\Phi_j(m)-V(x) m \ dx\\ 
\end{equation}
under the constraint $\int_{\Tt} m=1$; the constant $\Hh$ is the Lagrangian multiplier for the preceding constraint.

\section{First-order monotone MFGs}\label{sec: g_increasing}
        
We continue our analysis by considering monotonically increasing nonlinearities, $g$. In the case of a non-vanishing current, solutions are smooth. However, if the current vanishes,  solutions can fail to be smooth, $m$ can vanish, and $u$ may not be unique. 

The non-smooth behavior for a generic non-decreasing nonlinearity, $g$, was  observed in Theorem 2.8 in \cite{ll3} where the authors find limits of smooth solutions of second-order MFGs as the viscosity coefficient converges to 0.     
         
\subsection{$j\neq 0$, $g$ \bf increasing}
        
        Here, in contrast to the case $j=0$, examined later, the solutions are smooth. Elementary computations give the following result.
        
        \begin{proposition}\label{prp:j<>0g_increasing}
                Let $g$ be monotonically increasing. Then, for every $j>0$, \eqref{main} has a unique smooth solution, $(u_j,m_j,\Hh_j),$ with current $j$. This solution is given by
                \[
                m_j(x)= F_j^{-1}(\Hh_j-V(x)),\quad u_j(x)=\int\limits_{0}^{x} \frac{j}{m_j(y)}dy-p_j x,
                \] 
                where $p_j=\int\limits_{\Tt} \frac{j}{m_j(y)}dy,\ F_j(t)=\frac{j^2}{2t^2}-g(t),$ and $\Hh_j$ is such that $\int\limits_{\Tt} m_j(x)dx=1$.
        \end{proposition}
        
\subsection{$j=0$, $g$ \bf increasing}

To simplify the discussion and illustrate our methods, we consider \eqref{eq: currentform_j=0_intro} with $g(m)=m$. The analysis is similar for other choices of an increasing function, $g$. Accordingly, we have
\begin{equation}\label{eq: currentform_j=0_g=m}
\begin{cases}
\frac{(u_x+p)^2}{2}-m=\Hh-V(x);\\
m \geq 0,\ \int\limits_{\Tt} m dx=1;\\
m(u_x+p)=0.
\end{cases}
\end{equation}        
It is easy to see that
$
        m(x)=(V(x)-\Hh)^+$ for $ x \in \Tt
$.
 The map $\Hh \mapsto \int\limits_{\Tt} (V(x)-\Hh)^+ dx$ is decreasing (strictly decreasing at its positive values). Hence, there exists a unique number,  $\Hh$, such that $\int\limits_{\Tt} m(x)dx=1$. 
Moreover, $\Hh< \max V$ and  $\Hh\geq \int_\Tt V-1$.  If $\min V<\Hh<\max V$,  $m$ is non-smooth and there are regions where it vanishes. In contrast, if $\Hh<\min V$, $m$ is always positive. In this case, $u_x+p=0,$ and, by periodicity of $u$,  $p=0$.  Furthermore,  from the first equation in \eqref{eq: currentform_j=0_g=m}, we have
\[
\Hh=\int_\Tt V(x)dx-1. 
\]
Given $\Hh$, we find from \eqref{eq: currentform_j=0_g=m} that 
\begin{equation*}
        |u_x+p|=\sqrt{2(\Hh-V(x))^+},\quad x \in \Tt.
\end{equation*}
Hence, for
\begin{equation*}
        u^\pm(x)=\pm\int\limits_{0}^{x} \sqrt{2(\Hh-V(y))^+}dy-p x,
\end{equation*}
where $p=\pm \int\limits_{\Tt} \sqrt{2(\Hh-V(y))^+}dy$, the triplets $(u^\pm,m,\Hh)$ solve \eqref{eq: currentform_j=0_g=m}.
However, there are also solutions with a discontinuous derivative, $u_x$. For that, 
let $x_0 \in \Tt$ be such that $V(x_0)<\Hh$. Such a point always exists if $\Hh>\min \limits_{\Tt} V$ or, equivalently, when $\int\limits_{\Tt} V(x)dx -1 < \min\limits_{\Tt} V$. Let
\begin{equation*}
        (u^{x_0}(x))_x=\sqrt{2(\Hh-V(x))^+}\cdot\chi_{x<x_0}-\sqrt{2(\Hh-V(x))^+}\cdot\chi_{x>x_0}-p^{x_0},
\end{equation*}
where $p^{x_0}=\int\limits_{y<x_0} \sqrt{2(\Hh-V(y))^+}dy-\int\limits_{y>x_0} \sqrt{2(\Hh-V(y))^+}dy$ and $\chi$ denotes the characteristic function. 
Therefore, $u^{x_0}$ solves the first equation of \eqref{eq: currentform_j=0_g=m}  almost everywhere, and $u^{x_0}_x$ has  only negative jumps.  Since $m$ is continuous, $u^{x_0}$ is a viscosity solution of that  equation. Consequently, $(u^{x_0},m,\Hh)$ solves \eqref{eq: currentform_j=0_g=m}.

To summarize, \eqref{eq: currentform_j=0_g=m} has a unique, smooth solution if and only if $u_x+p \equiv 0$ or, equivalently, $m(x)=V(x)-\Hh$. The latter holds if and only if
\begin{equation}\label{eq: small_osc}
        \int\limits_{\Tt} V(x)dx \leq 1+\min \limits_{\Tt} V.
\end{equation}
This is the case for small perturbations $V$; that is, $\text{osc} V \leq 1$.

For $A\in \Rr$, set $V_{A}(x)=A \sin(2\pi (x+\frac 1 4))$ and let $m(x,A),\ \Hh(A)$ solve \eqref{eq: currentform_j=0_g=m} for $V=V_{A}$. In Fig.  \ref{fig: minterface}, we plot $m(x,A)$ for $0\leq A\leq 2$. We observe that $m(x,A)$ is smooth for small values of $A$ and becomes non-differentiable for large $A$, as expected from our analysis.  If $A=2$, \eqref{eq: small_osc} does not hold. Thus, $m(x,2)$ is singular and we have multiple solutions, $u(x,2)$. In Fig.  \ref{fig: us}, we plot $m(x,2)$ and two distinct solutions, $u(x,2)$.

        \begin{figure}
        \begin{center}
        \includegraphics[width=50mm]{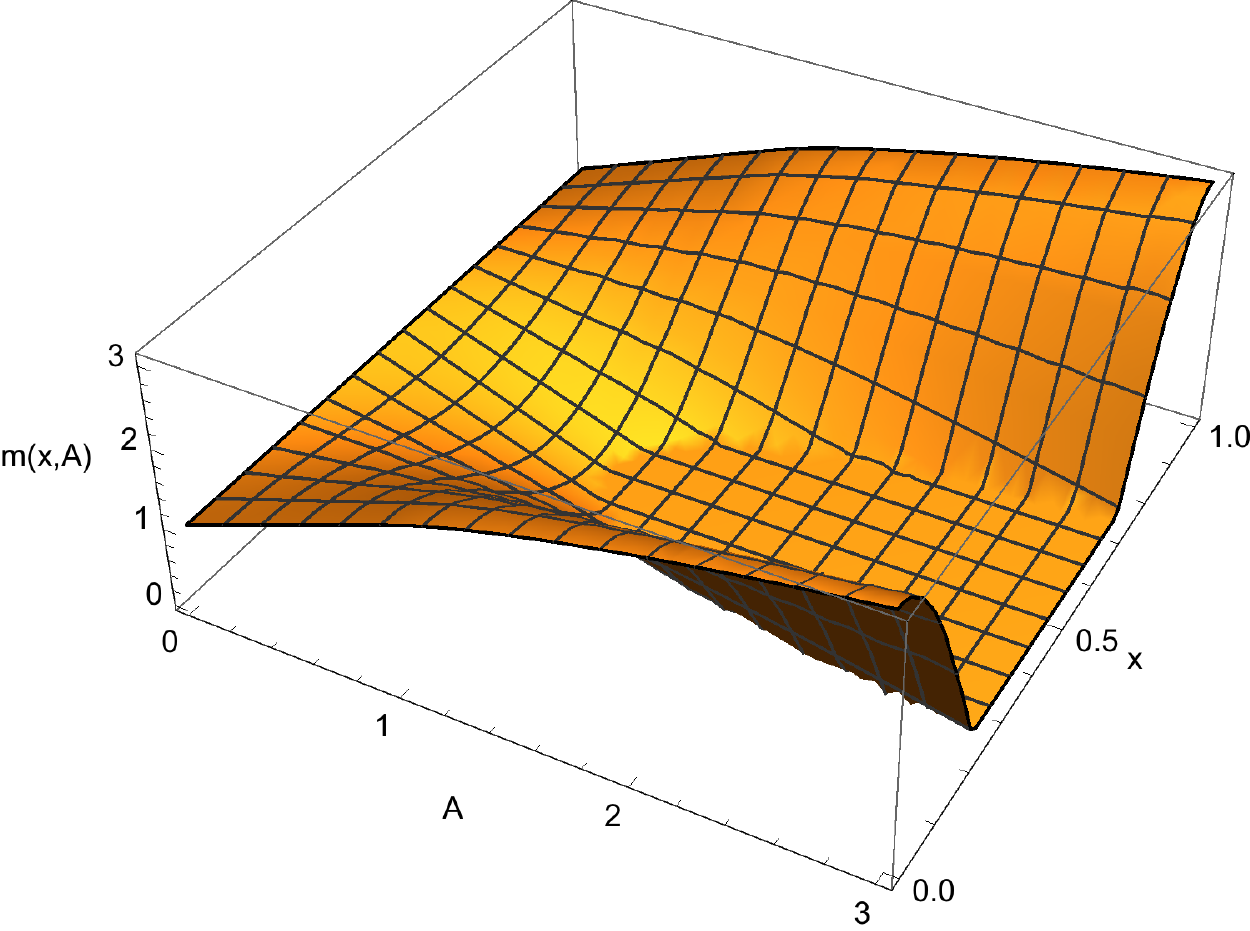}
        \caption{$m(x,A)$.}\label{fig: minterface}
        \end{center}
        \end{figure}

\begin{figure}
\includegraphics[width=50mm]{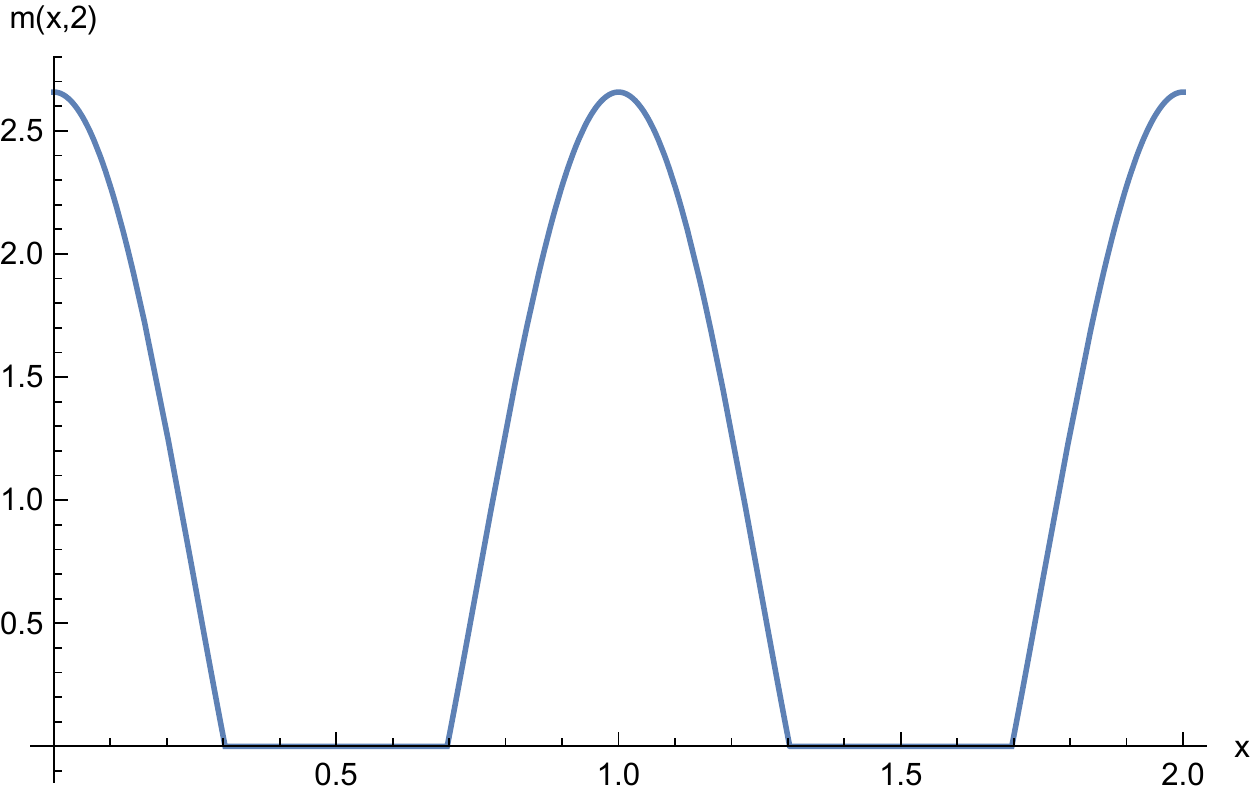}
\includegraphics[width=50mm]{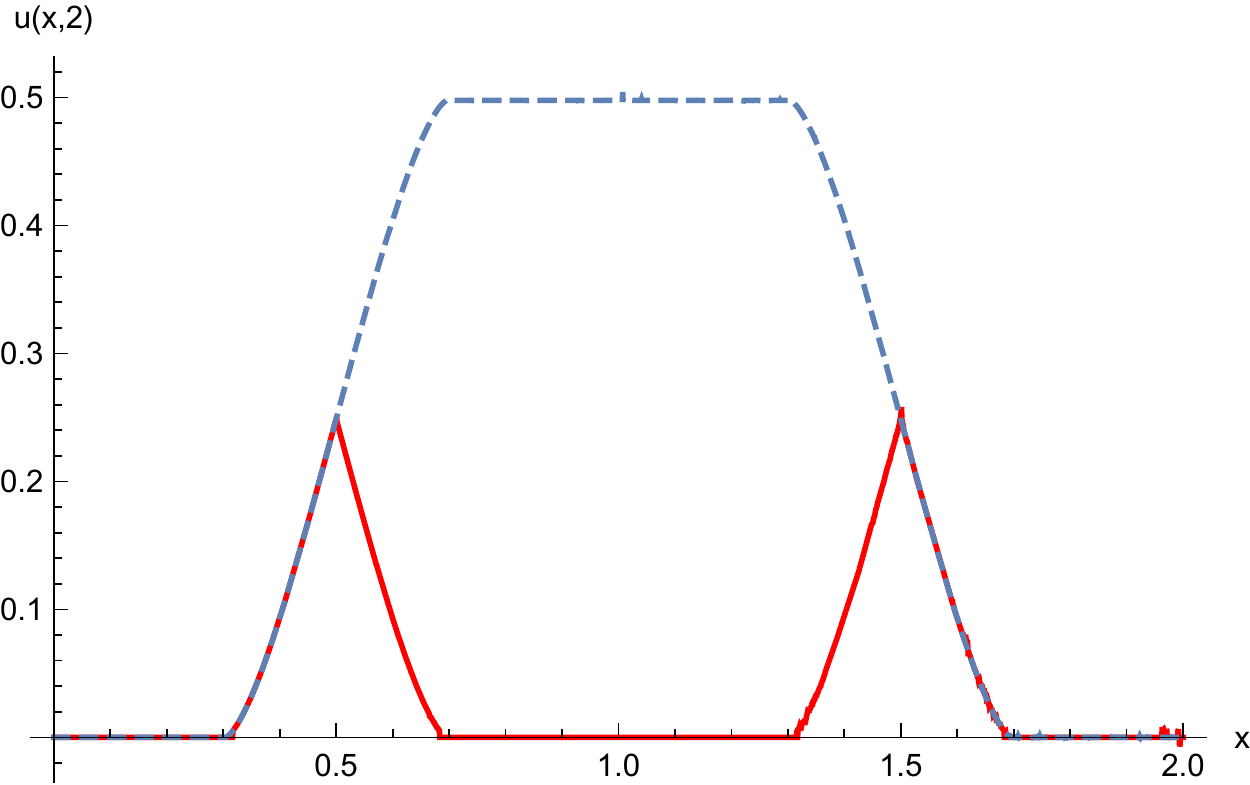}
\caption{$m(x,2)$ (left) and two distinct solutions $u(x,2)$ (right).  }\label{fig: us}
\end{figure}

\section{Monotone elliptic mean-field games}\label{sec:monelliptic}

To study \eqref{maine},  we examine the variational problem determined by  \eqref{mainef}. As before, for concreteness, we consider the case $g(m)=m$. In this case, \eqref{mainef} becomes
\begin{equation}
\label{mfun}
J_\epsilon [m]=
\int_{\Tt} \left(\epsilon^2 \frac{m_x^2}{2 m}+\frac{j^2}{2 m}+\frac{m^2}{2} -V(x) m\right) dx. 
\end{equation}
The preceding functional is convex and, as we prove next,   the direct method in the calculus of variations gives the existence of a minimizer on the set 
\[
\mathcal{A}=\left\{m\in W^{1,2}(\Tt):m\geq  0  \wedge  \int_{\Tt} m =1\right\}.
\]
\begin{pro}
\label{p41}
For each $j\in \Rr$, there exists a unique minimizer, $m$,  of $J_\epsilon[m]$ in $\Aa$. Moreover, $m>0$  and solves
\[
-\epsilon^2\left(\frac{m_x}m\right)_x-\frac{j^2}{2 m^2}+m+\Hh-V(x)=0 
\] 
for some constant $\Hh\in \Rr$. 
\end{pro}
\begin{proof}

The uniqueness of a positive minimizer is a consequence of the strict convexity of $J_\epsilon $. The existence of a non-negative minimizer requires separate arguments for the cases $j\neq 0$\ and $j=0$. 

We  first examine the case $j\neq 0$. We begin by taking a minimizing sequence, $m_n\in \Aa$. Then, there exists a constant, $C>0,$ such that 
\[
\int_{\Tt} \frac{(m_n)_x^2}{m_n}+\frac{1}{m_n}dx\leq C.
\]
Thus, by Morrey's theorem, the functions $\sqrt{m_n}$ are equi-H\"older 
continuous of exponent $\frac 1 2$.
Therefore, because $\int m_n=1$, this sequence is equibounded and, through some subsequence, $m_n\to m$ for some function $m\geq 0$. Moreover, by Fatou's lemma,
\[
\int_{\Tt} \frac 1 {m}dx\leq C.
\]
Suppose that  $\min m=m(x_0)=0$. Then, because $\sqrt m$ is H\"older continuous, we have $m(x)\leq C|x-x_0|$. However, 
\[
\int_{\Tt} \frac{1}{|x-x_0|}dx
\]
is not finite, which is a contradiction. Thus,   $m$ is a strictly positive minimizer. Moreover,  it solves the corresponding Euler-Lagrange equation.

For  $j=0$, we rewrite the Euler-Lagrange equation as
\begin{equation}
\label{elee1}
-\epsilon^2(\ln m)_{xx} + m -V(x) = -\Hh.
\end{equation}

Let $\mathcal{P}$ be the set of non-negative functions in $L^\infty(\Tt^d)$
and consider the map $\Xi:\mathcal{P}\to  \mathcal{P}$ defined as follows.
Given $\eta\in \mathcal{P}$, we solve the PDE
\[
-\epsilon^2w_{xx} + \eta -V(x) = -\Hh, 
\]
where $\Hh$ satisfies the 
compatibility condition
\[
\Hh = \int_{\Tt} V dx - 1,  
\]
and $w:\Tt\to \Rr$ is such that $\int e^w dx=1$. An elementary argument shows
that $w$ is uniformly bounded from above and from below. Next,
we set $\Xi(\eta)=e^w$. The mapping $\Xi$ is continuous and compact. Accordingly,
by Schauder's Fixed Point Theorem, there is a fixed point, $m,$\, that solves
\eqref{elee1}. By the convexity of the variational problem \eqref{mfun},   this fixed point is the unique solution of the Euler-Lagrange equation.  \end{proof}

Next, to study the convergence as $\epsilon\to 0$, we investigate the 
$\Gamma$-convergence as $\epsilon\to 0$ of  $J_\epsilon$. A simple modification of the arguments in \cite{Braides}, Chapter 6, shows that 
\[
J_\epsilon \overset{\Gamma}{\to} J, 
\]
where
\[
J[m]=
\int_{\Tt} \left(\frac{j^2}{2 m}+\frac{m^2}{2}
-V(x) m\right) dx, 
\]
if $m\geq 0$ and $\int m=1$. In Fig. \ref{fig: eliptic1}, we observe numerical evidence for this $\Gamma$-convergence.

\begin{figure}
        \centering      
        \includegraphics[width=50mm]{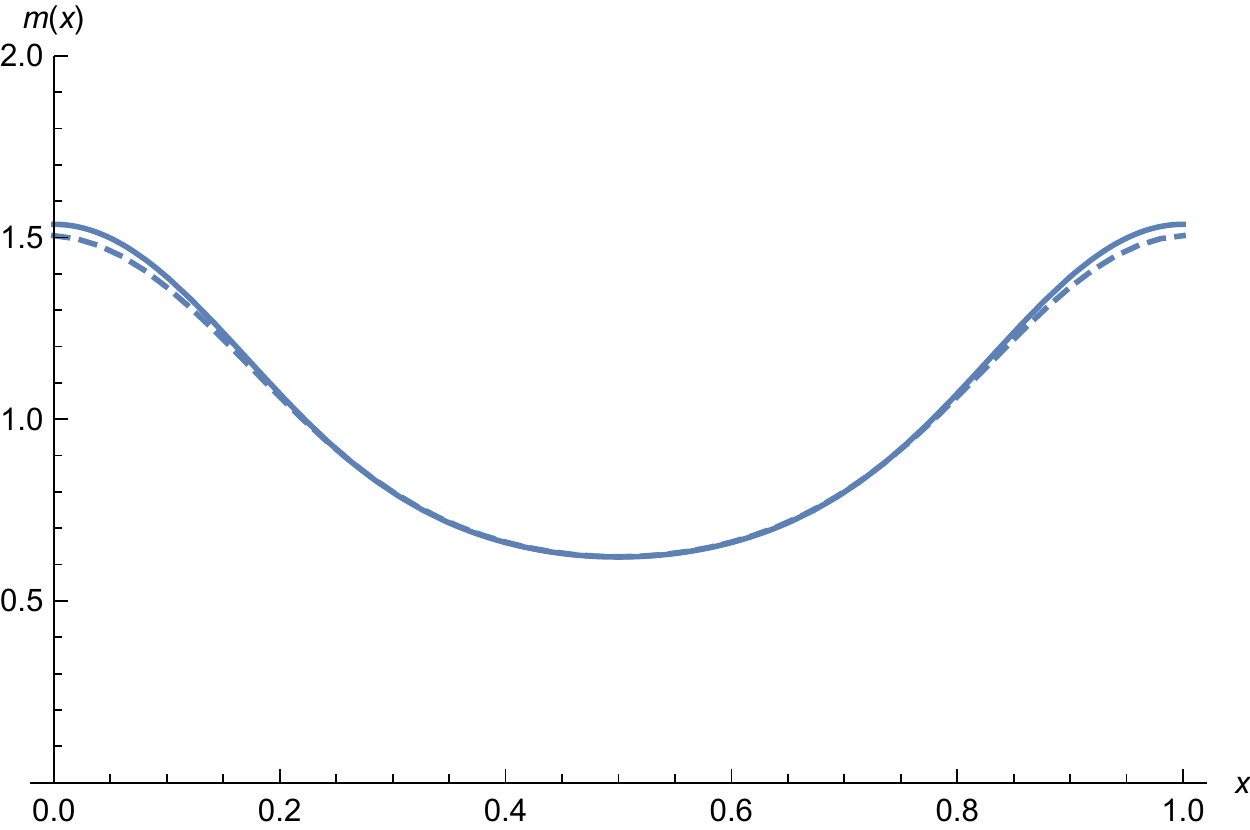}
        \caption{ Solution $m$ of \eqref{maine} when $g(m)=m,\ j=1,\ V(x)= \sin(2 \pi (x + 1/4))$ for $\epsilon=0.01$ (dashed) and for $\epsilon=0$ (solid).}
        \label{fig: eliptic1}
\end{figure}

\section{Regular viscosity solutions in anti-monotone mean-field games}   \label{sec: g_decreasing}
 
Here, we investigate MFGs with decreasing $g$. 
To simplify, we assume that $g(m)=-m$. However, our arguments are valid for a general decreasing $g$.
In contrast with the monotone case, $m$ may not be unique.
Furthermore,  $m$ can be discontinuous and, thus, viscosity solutions of the Hamilton-Jacobi equation in \eqref{main} should be interpreted in the discontinuous sense. In this section, we are interested in regular discontinuous viscosity solutions; that is, solutions satisfying conditions a. and b. stated in  the Introduction. Here, we examine existence, uniqueness, and additional properties of such solutions. In Section \ref{discsec}, we prove that these solutions are indeed discontinuous viscosity solutions.

\subsection{$j \neq 0$, $g$ decreasing}\label{sec: j>0}

To simplify the presentation, we consider $j>0$.

        With $g(m)=-m,$ \eqref{eq: currentform} becomes
        \begin{equation}\label{eq: currentform_1}
        \begin{cases}
        \frac{j^2}{2m^2}+m=\Hh-V(x);\\
        m>0,\ \int\limits_{\Tt} m dx=1;\\
        \int\limits_{\Tt} \frac{1}{m} dx=\frac{p}{j}.
        \end{cases}
        \end{equation}
        The minimum of $t\mapsto j^2/2t^2+t$ is attained at $t_{min}=j^{2/3}$. Thus, $j^2/2t^2+t \geq 3j^{2/3}/2$ for $t>0$.

Therefore, a lower bound for $\Hh$ is
        \begin{equation}\label{eq: Hboundbybelow}
        \Hh \geq \Hh_j^{cr} =\max_{\Tt} V+\frac{3j^{2/3}}{2},
        \end{equation}
        where the superscript $^{cr}$ stands for critical.
        
        The function $t \mapsto j^2/2t^2+t$ is decreasing on the interval $(0,t_{min})$ and increasing on the interval $(t_{min},+\infty)$. For any $\Hh$ satisfying \eqref{eq: Hboundbybelow}, let $m_{\Hh}^{-}$ and $m_{\Hh}^{+}$ be the solutions of 
\[
\frac{j^2}{2(m_{\Hh}^\pm(x))^2}+m_{\Hh}^\pm(x)=\Hh-V(x),
\]
with $0\leq  m_{\Hh}^-(x)\leq  t_{min}\leq  m_{\Hh}^+(x)$.
        Due to \eqref{eq: Hboundbybelow},  $m_{\Hh}^{-}$ and $m_{\Hh}^{+}$ are well defined. Furthermore, if $(u,m , \Hh)$ solves \eqref{main}, then $m(x)$ agrees with either $m^+_{\Hh}(x)$ or $m^-_{\Hh}(x)$, almost everywhere in $\Tt$.
        
        Let $m_{j}^-:=m_{\Hh^{cr}_j}^{-}$ and $m_{j}^+:=m_{\Hh^{cr}_j}^{+}$. Note that $m_{j}^-(x)\leq m_{j}^+(x)$ for all $x \in \Tt$, and the equality holds only at the maximum points of $V$. Hence, $m_{j}^-(x)< m_{j}^+(x)$ on a set of positive Lebesgue measure unless $V$ is constant.
        
        The two fundamental quantities for our analysis are
        \begin{equation}\label{eq: a+-}
                \begin{cases}
                        \alpha^{+}(j)=\int\limits_{0}^{1}m_{j}^{+}(x)dx,\\
                        \alpha^{-}(j)=\int\limits_{0}^{1}m_{j}^{-}(x)dx.
                \end{cases}
        \end{equation}
        If $V$ is not constant, we have
        \begin{equation*}
                \alpha^{-}(j)<\alpha^{+}(j)
        \end{equation*}
        for $j>0$.
        
        \begin{proposition}\label{prp: j>0}
                Suppose that $x=0$ is the single maximum of $V$. Then, for every $j>0,$ there exists a unique number, $p_j$, such that \eqref{main} has a regular solution with a current level, $j$. Moreover, the solution of \eqref{eq: currentform_1}, $(u_j,m_j,\Hh_j)$, is unique and given as follows.
                \begin{itemize}
                        \item[i.] If $\alpha^+(j) \leq 1,$
                        \begin{equation}\label{eq: sols_i}
                        m_j(x)=m^{+}_{\Hh_j}(x),\quad u_j(x)=\int\limits_{0}^{x}\frac{jdy}{m_j(y)}-p_jx,
                        \end{equation}
                        where $p_j=\int\limits_{\Tt}\frac{jdy}{m_j(y)}$ and $\Hh_j$ is such that $\int\limits_{\Tt} m_j(x)dx=1$.
                        \item[ii.] If $\alpha^-(j) \geq 1,$

                        \begin{equation}\label{eq: sols_ii}
                        m_j(x)=m^{-}_{\Hh_j}(x),\quad u_j(x)=\int\limits_{0}^{x}\frac{jdy}{m_j(y)}-p_jx,
                        \end{equation}
                        where $p_j=\int\limits_{\Tt}\frac{jdy}{m_j(y)}$ and $\Hh_j$ is such that $\int\limits_{\Tt} m_j(x)dx=1$.
                        \item[iii.] If $\alpha^-(j) < 1 < \alpha^+(j)$, we have that $\Hh_j=\Hh_j^{cr}$, and                         \begin{equation}\label{eq: sols_iii}
                        m_j(x)=m^{-}_{j}(x)\chi_{[0,d_j)}+m^{+}_{j}(x)\chi_{[d_j,1)},\ u_j(x)=\int\limits_{0}^{x}\frac{jdy}{m_j(y)}-p_jx,
                        \end{equation}
                        where $p_j=\int\limits_{\Tt}\frac{jdy}{m_j(y)}$ and $d_j$ is such  that \[\int\limits_\Tt m_j(x)dx=\int\limits_{0}^{d_j} m^-_{j}(x)dx+\int\limits_{d_j}^1 m^{+}_{j}(x)dx=1.\]
                \end{itemize}
        \end{proposition}
        \begin{proof}
        \textbf{Case i.} The function $j^2/2t^2+t$ is increasing on the interval $(t_{min},+\infty)$. Therefore,  $\Hh \mapsto m_{\Hh}^+(x)$ is increasing for all $x.$ Hence,  the mapping
        \begin{equation*}
        \Hh\mapsto \int\limits_{\Tt}m_{\Hh}^{+}(x)dx,
        \end{equation*}
        is increasing. By assumption,
        \(
                \int\limits_{\Tt}m_{\Hh_j^{cr}}^{+}(x)dx=\int\limits_{\Tt}m_j^+(x)dx\leq 1.
        \)
        Therefore, there exists a unique $\Hh_j \geq \Hh_j^{cr}$ such that $\int\limits_{\Tt}m_{\Hh_j}^{+}(x)dx=1$. Thus,  $(u_j,m_j)$ given by  \eqref{eq: sols_i} is the unique solution of \eqref{main} with $\Hh=\Hh_j$ and $p=p_j$.
        
        \textbf{Case ii.} The function $j^2/2t^2+t$ is decreasing on the interval $(0,t_{min})$. Therefore,  $m_{\Hh}^-(x)$ is decreasing in $\Hh$ for all $x$. Hence, the mapping
        \begin{equation*}
        \Hh\mapsto \int\limits_{\Tt}m_{\Hh}^{-}(x)dx
        \end{equation*}
        is decreasing. By assumption,
        \(
        \int\limits_{\Tt}m_{\Hh_j^{cr}}^{-}(x)dx=\int\limits_{\Tt}m_j^-(x)dx\geq 1.
        \)
        Thus, there exists a unique number,  $\Hh_j \geq \Hh_j^{cr}$, such that $\int\limits_{\Tt}m_{\Hh_j}^{-}(x)dx=1$. Hence, $(u_j,m_j)$ given by \eqref{eq: sols_ii} is the unique solution of \eqref{main} with $\Hh=\Hh_j$ and $p=p_j$.
        
        \textbf{Case iii.} We first show that \eqref{main} does not have regular solutions for $\Hh>\Hh_j^{cr}$. By contradiction, suppose that \eqref{main} has a regular solution, $(u,m,\Hh),$ for some $\Hh>\Hh_j^{cr}$ and $p \in \Rr$.
Evidently, $
        m(x)=m_{\Hh}^+(x)\chi_E+m_{\Hh}^-(x)\chi_{\Tt\setminus E}
       $
        for some subset $E \subset \Tt$. Furthermore,         \begin{equation}\label{eq: gap}
                \inf_{\Tt} (m_{\Hh}^+(x)-m_{\Hh}^-(x))>0
        \end{equation}
        because $\Hh>\Hh_j^{cr}$. Moreover,
        \[\int\limits_{\Tt}m(x)dx=\int\limits_{E}m_{\Hh}^+(x)dx+\int\limits_{\Tt\setminus E}m_{\Hh}^-(x)dx
        \]
        and        \[\int\limits_{\Tt}m^-_{\Hh}(x)dx<\int\limits_{\Tt}m^-_{j}(x)dx<1<\int\limits_{\Tt}m^+_{j}(x)dx<\int\limits_{\Tt}m^+_{\Hh}(x)dx.
        \]
        Therefore, neither $E$ nor $\Tt\setminus E$ can be empty or have zero Lebesgue measure.
Because $E$ and $\Tt\setminus E$ are not negligible, there exists a real number, $e,$
such that for every $\epsi>0,$
        \begin{equation*}
        (e-\epsi,e) \cup E \neq \emptyset \quad \text{and}\quad (e,e+\epsi) \cup E^c \neq \emptyset.
        \end{equation*}
        
        According to \eqref{eq: gap}, $m$ has a negative jump, $m(e^-)-m(e^+)<0$, at $x=e$. Hence, $u_x=j/m-p$ has a positive jump, $\frac{j}{m^-(e)}-\frac{j}{m^+(e)}>0$, at $x=e$. However, derivatives of regular solutions can only have negative jumps and, thus, this contradiction implies $\Hh_j=\Hh^{cr}_j$.
        
        Next, we construct  $m_j$ and $u_j$ and determine $p_j$. We look for a function  $m_j$ of the form
        \begin{equation}\label{eq: m_d}
        m_j(x)=\begin{cases}
        m^-_{j}(x),\ x\in [0,d),\\
        m^+_{j}(x),\ x\in [d,1).
        \end{cases}
        \end{equation}
        Note that \eqref{eq: m_d} is the only possibility for $m_j$ because $m_j$  can switch from $m^+_j$ to $m^-_j$ only if there is no jump at the switching point; that is, $m^+_j$ and $m^-_j$ are equal at that point,  which only holds at maximum of $V$. Thus, by periodicity, $m_j$ can switch to $m^-_j$ from $m^+_j$ only at $x=0$ and $x=1$.

        It remains to choose $d \in (0,1)$ such that $\int\limits_{\Tt} m_j(x)dx=1$. Let
        \[\phi(d)=\int\limits_{0}^1 m_j(x)dx=\int\limits_{0}^d m^-_{j}(x)dx+\int\limits_{d}^1 m^{+}_{j}(x)dx.
        \]
        Because $\phi(0)>1$ and $\phi(1)<1$
        and because $\phi'(d)=m^-_{j}(d)-m^+_{j}(d)<0$ for $d\in (0,1)$, there exists a unique $d_j \in (0,1)$ such that $\phi(d_j)=1$. The triplet defined by \eqref{eq: sols_iii},
$(u_j,m_j,\Hh_{j}),$ solves  \eqref{main}.
        \end{proof}
        By the previous proposition, if $V$ has a single maximum point then, for every current, $j>0$, there exists a unique $p_j$ and a unique triplet, $(u_j,m_j,\Hh_j)$, that solves \eqref{eq: currentform_1} for $p=p_j$.
In contrast, as we show next, if $V$ has multiple maxima and $j>0$ is such that Case \textit{iii} in Proposition \ref{prp: j>0} holds, there exist infinitely many solutions.
        
        \begin{proposition}\label{prp: V_multimax}
                Suppose that $V$ attains a maximum at $x=0$ and at $x=x_0 \in (0,1)$. Let $j$ be such that $\alpha^-(j) < 1 < \alpha^+(j)$. Then, there exist infinitely many numbers, $p,$ and pairs, $(u,m),$ such that $(u,m,\Hh_j^{cr})$ is a regular solution of \eqref{main}.
        \end{proposition}
        \begin{proof}We look for solutions of the form
                \begin{equation*}
                m_j^{d_1,d_2}(x)=\begin{cases}
                m^-_{j}(x),\ x\in [0,d_1) \cup [x_0,d_1),\\
                m^+_{j}(x),\ x\in [d_1,x_0)\cup [d_2,1),
                \end{cases}
                \end{equation*}
                where $0<d_1<x_0$ and $x_0<d_2<1$. Note that $m_j^{d_1,d_2}$ has two discontinuity points. At these points,  $m_j^{d_1,d_2}$
 has positive jumps. Hence, if we define
                \begin{equation*}
                u_j^{d_1,d_2}(x)=\int\limits_{0}^{x}\frac{jdy}{m_j^{d_1,d_2}(y)}-p_j^{d_1,d_2}x,\quad x \in \Tt,
                \end{equation*}
                where $p_j^{d_1,d_2}=\int\limits_{\Tt}\frac{jdy}{m_{d_1,d_2}(y)}$, the triplet $(u_j^{d_1,d_2},m_j^{d_1,d_2},\Hh_j^{cr})$ is a regular solution of \eqref{main} if \[\int\limits_{\Tt}m_j^{d_1,d_2}(x)dx=1.\]
                
                To determine $d_1$ and $d_2$, we consider the function
                \begin{align*}
                \phi(d_1,d_2)&=\int\limits_{0}^1 m_j^{d_1,d_2}(x)dx=\int\limits_{0}^{d_1} m^-_{j}(x)dx+\int\limits_{d_1}^{x_0} m^{+}_{j}(x)dx\\ \nonumber
                &+\int\limits_{x_0}^{d_2} m^-_{j}(x)dx+\int\limits_{d_2}^{1} m^{+}_{j}(x)dx, \quad (d_1,d_2) \in (0,x_0)\times (x_0,1).
                \end{align*}
                We have that $\phi(0,x_0)=\int\limits_{0}^1 m^{+}_{j}(x)dx>1$ and $\phi(x_0,1)=\int\limits_{0}^1 m^{-}_{j}(x)dx<1$. Because $\phi$ is continuous, there exists a pair, $(d_1,d_2)\in (0,x_0)\times (x_0,1),$ such that $\phi(d_1,d_2)=1$. In fact, there are infinitely many such pairs. For arbitrary continuous curve $\gamma \subset [0,x_0]\times [x_0,1]$ connecting the points $(0,x_0)$ and $(x_0,1)$, there exists at least one pair, $(d_1,d_2) \in \gamma,$ such that $\phi(d_1,d_2)=1$. To each such pair corresponds a triplet $(u_j^{d_1,d_2},m_j^{d_1,d_2},\Hh_j^{cr})$ that is a regular solution of \eqref{main}.
                \end{proof}

\begin{figure}
        \centering      
        \includegraphics[width=50mm]{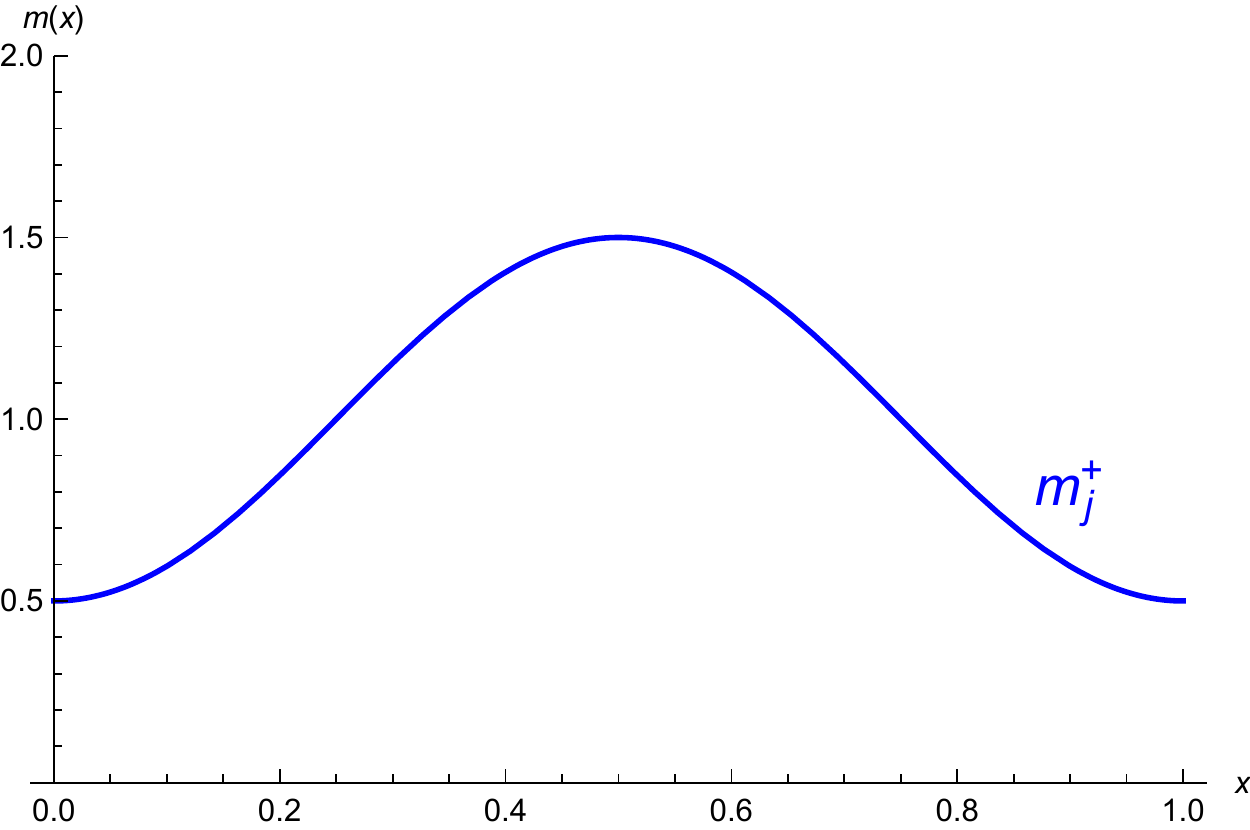}
        \caption{Solution $m$ for $j=0.001$ and $V(x)=\frac 1 2 \sin(2 \pi (x + 1/4))$.}
        \label{fig: casei}
\end{figure}

\begin{figure}
        \centering      
        \includegraphics[width=50mm]{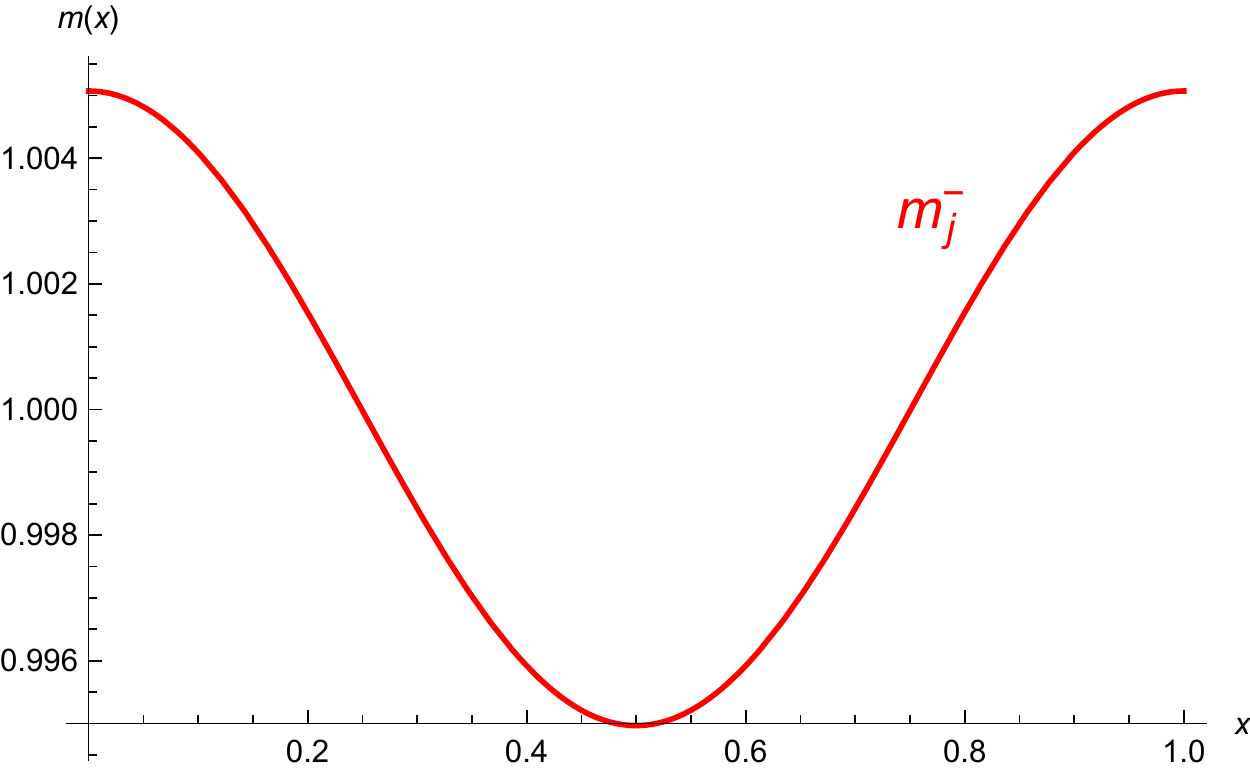}
        \caption{Solution $m$ for $j=10$ and $V(x)=\frac 1 2 \sin(2 \pi (x + 1/4))$.}
        \label{fig: caseii}
\end{figure}

\begin{figure}
        \centering      
        \includegraphics[width=50mm]{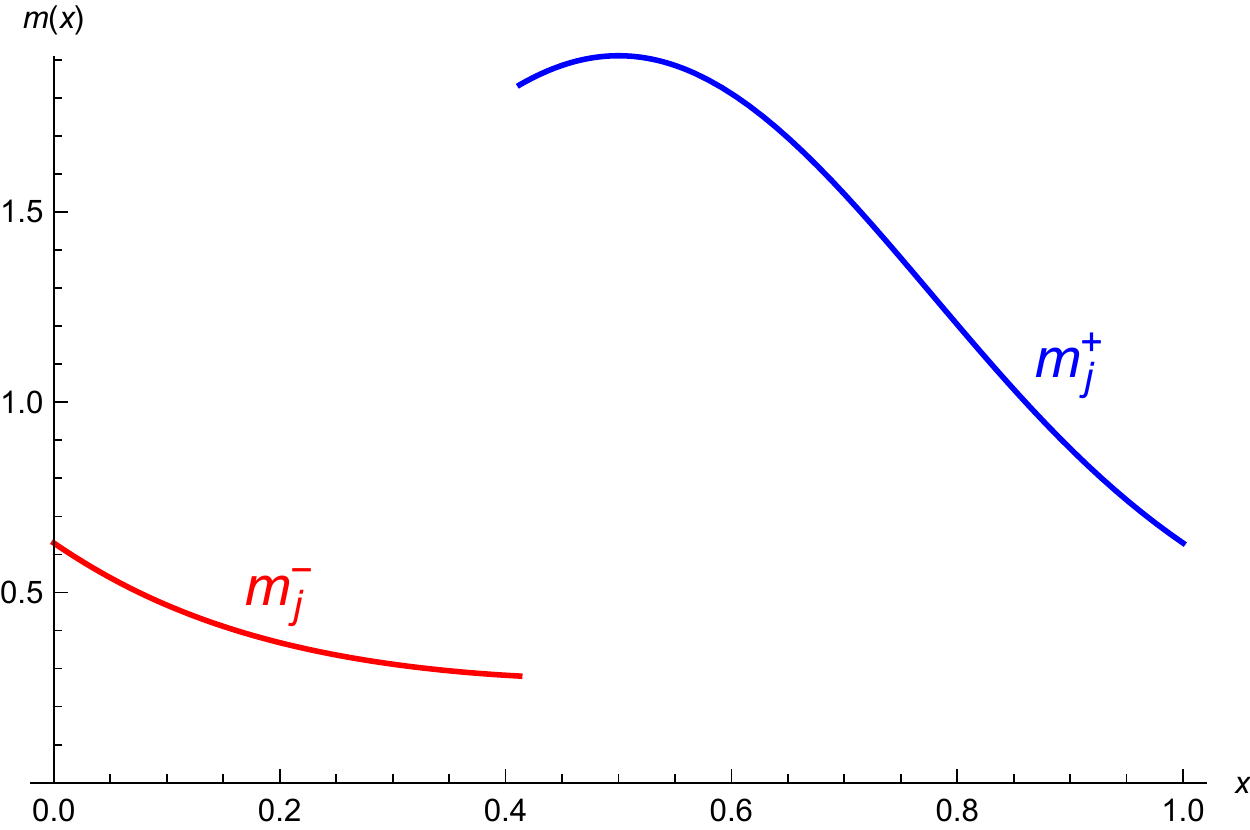}
        \caption{Solution $m_j$ for $j=0.5$ and $V(x)=\frac 1 2 \sin(2 \pi (x + 1/4))$.}
        \label{fig: caseiii}
\end{figure}

\begin{figure}
        \centering      
        \includegraphics[width=50mm]{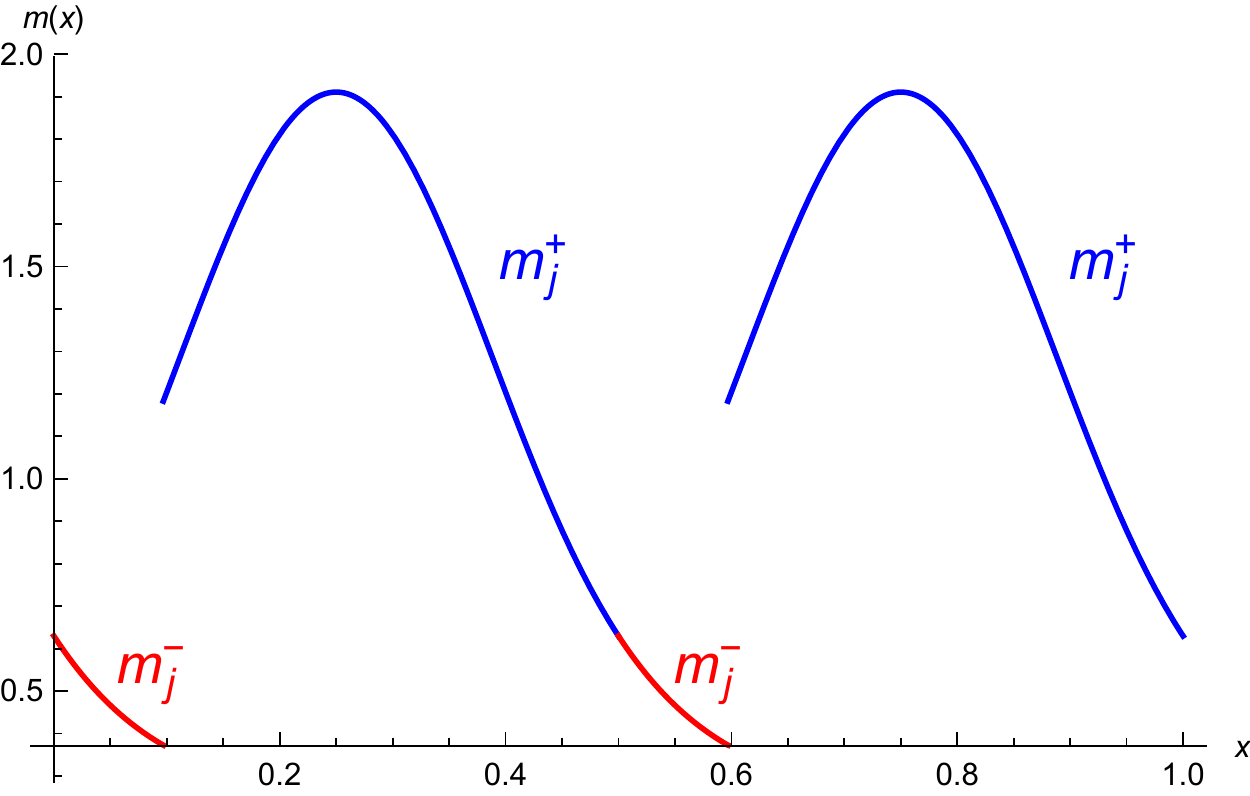}
        \mbox{ } \mbox{ }\mbox{ }\mbox{ }\mbox{ }\mbox{ }
        \includegraphics[width=50mm]{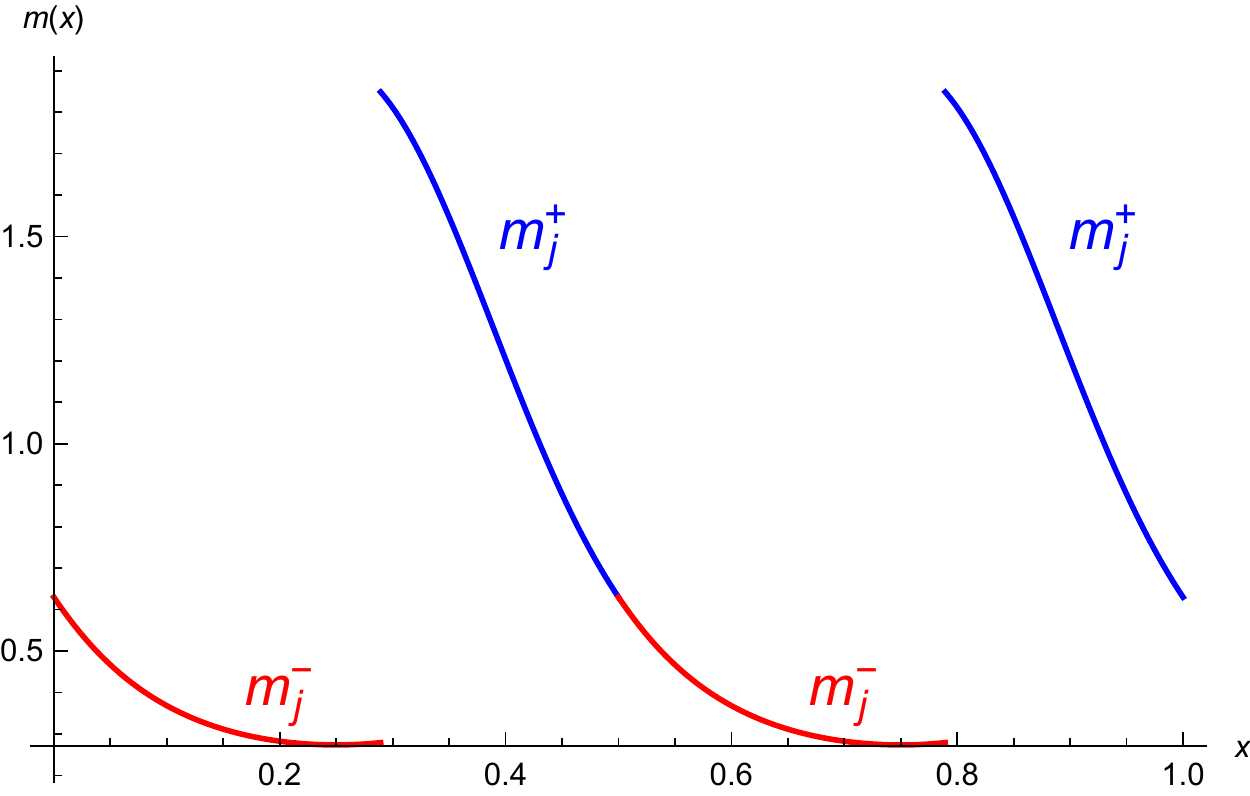}
        \caption{Two distinct solutions for $j=0.5$ and $V(x)=\frac 1 2 \sin(4 \pi (x + 1/8))$.}
        \label{fig: multimax}
\end{figure}

Let $V(x)=\frac{1}{2}\sin(2\pi(x+\frac{1}{4}))$. Because $V$ has a single maximum, Proposition \ref{prp: j>0} gives that \eqref{main} admits a unique regular solution for all values of $j>0$. In Figs. \ref{fig: casei}, \ref{fig: caseii}, \ref{fig: caseiii}, we plot $m$ for different values of $j$.
In Fig. \ref{fig: casei}, we plot $m$ in the low-current regime, $j=0.001;$ that is,  Case i in Proposition \ref{prp: j>0}. As we can see, $m$ is smooth as predicted by the proposition.
In Fig. \ref{fig: caseii}, we plot $m$ in the high-current regime, $j=10; $ that is,  Case ii in Proposition \ref{prp: j>0}. As before, we observe that $m$ is smooth.
Finally, 
in Fig. \ref{fig: caseiii}, we plot $m$ for the intermediate-current regime, $j=0.5$; that is,   Case iii in Proposition \ref{prp: j>0}. As we can see, $m$ is discontinuous.

Next, we consider the potential $V(x)=\frac{1}{2}\sin(4\pi(x+\frac{1}{8})) $ that has two maxima. By Proposition \ref{prp: V_multimax}, we have infinitely many two-jump solutions. In Fig. \ref{fig: multimax}, we plot two such solutions. 
        
\subsection{$j = 0$, $g$ decreasing}\label{sec: j=0}
        
        Now, we examine the case when the current vanishes, and, thus,  we consider the system
        \begin{equation}\label{eq: currentform_j=0}
        \begin{cases}
        \frac{(u_x+p)^2}{2}+m=\Hh-V(x);\\
        m \geq 0,\ \int\limits_{\Tt} m dx=1;\\
        m(u_x+p)=0.
        \end{cases}
        \end{equation}  
        Suppose that  \eqref{eq: currentform_j=0} has a solution.  Because $m \geq 0$, we have $\Hh-V(x)\geq~0$ for$\ x\in~\Tt$. Thus, $\Hh\geq \max\limits_{\Tt} V$. On the other hand,
        \[\int\limits_{\Tt} \left(\Hh-V(x)\right) dx \geq \int\limits_{\Tt} m dx=1.
        \]
        Consequently, $\Hh\geq 1+ \int\limits_{\Tt} V$. Therefore,
        \begin{equation*}
        \Hh\geq \max \left(\max\limits_{\Tt} V,1+\int \limits_{\Tt}V\right)=:\Hh_0.
        \end{equation*}
        It turns out that $\Hh_0$ is the only possible value for $\Hh$ as we show next.
        \begin{proposition}\label{prp: H_0}
                The MFG \eqref{eq: currentform_j=0} does not have regular solutions for $\Hh>\Hh_0$. 
        \end{proposition}
        \begin{proof}
                Suppose that $\Hh>\Hh_0$ and that the triplet $(u,m,\Hh)$ is a regular solution of \eqref{eq: currentform_j=0}.
If $m(x)>0$, then $u_x(x)+p=0$ and $m=\Hh-V(x)$.  If $m(x)=0$, then $$(u_x(x)+p)^2=2(\Hh-V(x)).$$ Thus, on the set 
$Z=\{x:m(x)=0\}$,
$$u_x(x)+p=\sqrt{2(\Hh-V(x))}\ \text{or}\ u_x(x)+p=-\sqrt{2(\Hh-V(x))}.$$ We have that $\int\limits_{\Tt} \left(\Hh-V(x)\right) dx>1$. Hence, the set $Z$ has a positive Lebesgue measure. Otherwise, $m(x)=\Hh-V(x)$ everywhere, and thus $\int \limits_{\Tt}m(x)dx>1$. Consequently, $u_x+p$ is either $\sqrt{2(\Hh-V(x))}$ or $-\sqrt{2(\Hh-V(x))}$ on $Z$.
Suppose that $u_x(x)+p$ takes the value $-\sqrt{2(\Hh-V(x))}$ at some point $x \in \Tt$. Without loss of generality, we can assume that $u_x(0)+p=-\sqrt{2(\Hh-V(0))}$. Let \[e=\sup \left\{x \in (0,1)\ \text{s.t.}\  u_x(x)+p=-\sqrt{2(\Hh-V(x))}\right\}.\]
                
                Then, at  $x=e$, the function $u_x+p$ has a jump of size $\sqrt{2(\Hh-V(e))}$ or $2\sqrt{2(\Hh-V(e))}$. However, this is impossible because $u_x$ is a regular solution, and it cannot have positive jumps.
Therefore, $u_x(x)+p$ takes only the values $\sqrt{2(\Hh-V(x))}$ and $0$. But then, $u_x$ must have a positive jump from  $0$ to  $\sqrt{2(\Hh-V(x))}$ at some point, which  also contradicts the regularity property.
                \end{proof}

        Now, we construct solutions to \eqref{eq: currentform_j=0} with $\Hh=\Hh_{0}$. It turns out that if $V$ has a large oscillation, then \eqref{eq: currentform_j=0} has infinitely many regular solutions.
        \begin{proposition}\label{prp: j=0} We have that
                \begin{itemize}
                        \item[i.] if $1+\int \limits_{\Tt}V \geq \max\limits_{\Tt} V$, then the triplet $(u_0,m_0,\Hh_0)$ with
                         \begin{equation}\label{eq: sols_i_j=0}
                                m_0(x)=\Hh_0-V(x),\ u_0(x)=0,
                        \end{equation}
                        solves \eqref{eq: currentform_j=0} in the classical sense for $p=0$;
                        \item[ii.] if $\max\limits_{\Tt} V > 1+\int \limits_{\Tt}V$, define
                        \begin{equation}\label{eq: m_d1d2}
                        m_0^{d_1,d_2}(x)=\begin{cases}
                        \Hh_{0}-V(x),\ x\in [d_1,d_2] ,\\
                        0,\ x\in \Tt \setminus [d_1,d_2],
                        \end{cases}
                        \end{equation}
                        and
                        \begin{equation}\label{eq: u_d1d2}
                        u_0^{d_1,d_2}(x)=\int\limits_{0}^{x} (u_0^{d_1,d_2})_x(y) dy,\quad x\in \Tt,
                        \end{equation}
                        where 
                        \begin{equation*}
                        (u_0^{d_1,d_2})_x(x)=\begin{cases}
                        \sqrt{2(\Hh_{0}-V(x))}-p_0^{d_1,d_2},\ x\in [0,d_1) ,\\
                        -p_0^{d_1,d_2},\ x\in [d_1,d_2],\\
                        -\sqrt{2(\Hh_{0}-V(x))}-p_0^{d_1,d_2},\ x\in (d_2,1],
                        \end{cases}
                        \end{equation*}
                        and $p_0^{d_1,d_2}=\int\limits_{0}^{d_1} \sqrt{2(\Hh_{0}-V(x))}dx-\int\limits_{d_2}^{1} \sqrt{2(\Hh_{0}-V(x))}dx$. Then, for any pair, $(d_1,d_2)$, such that
                        \begin{equation}\label{eq: int_of_H0}
                                \int\limits_{d_1}^{d_2} (\Hh_{0}-V(x))dx=1,
                        \end{equation}
                        the triplet $(u_0^{d_1,d_2},m_0^{d_1,d_2},\Hh_0)$ is a regular solution for \eqref{eq: currentform_j=0} for $p=p_0^{d_1,d_2}$. Furthermore, there exist infinitely many pairs, $(d_1,d_2)$, such that \eqref{eq: int_of_H0} holds.
                \end{itemize}
        \end{proposition}
        \begin{proof}
                \textbf{Case i.} In this case, $\Hh_0=1+\int\limits_{\Tt} V(x)dx$ and straightforward computations show that \eqref{eq: sols_i_j=0} defines a classical solution of \eqref{eq: currentform_j=0}.
                
                \textbf{Case ii.} In this case, we have that $\Hh_{0}=\max\limits_{\Tt} V$ and that $\int\limits_{0}^{1} (\Hh_{0}-V(x))dx>1$. Without loss of generality, we assume that $0$ is a point of maximum for $V$.
                
                Note that $(u_0^{d_1,d_2})_x$ has only negative jumps and $u_0^{d_1,d_2}$ satisfies \eqref{eq: currentform_j=0} almost everywhere. Thus, the triplet $(u_0^{d_1,d_2},m_0^{d_1,d_2},\Hh_0)$ is a regular solution of \eqref{eq: currentform_j=0} if $\int\limits_{\Tt}m_0^{d_1,d_2}(x)dx=1$. However, the latter is equivalent to \eqref{eq: int_of_H0}. Since $\int\limits_{0}^{1} (\Hh_{0}-V(x))dx>1,$ we can find infinitely many such pairs. We find $p_0^{d_1,d_2}$ from the identity $\int\limits_{\Tt} (u_0^{d_1,d_2})_x(x)dx=0$.
        \end{proof}
        
       Figs. \ref{pic: m0} and \ref{pic: u0} show the solutions of \eqref{eq: currentform_j=0} for $V(x)=5\sin(2\pi(x+\frac 1 4) ),\ x \in \Tt$.

        \begin{figure}
                \centering      
                \includegraphics[width=50mm]{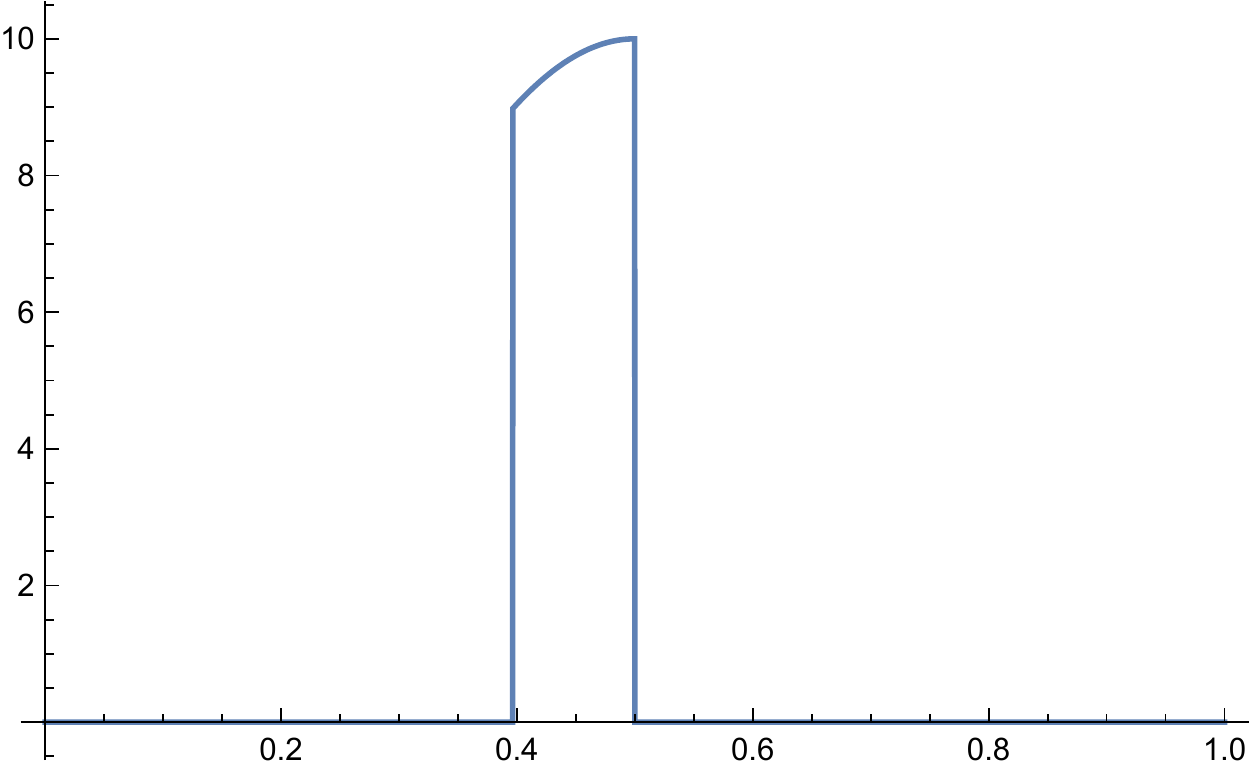}
                \caption{$m_0$ as defined in \eqref{eq: m_d1d2} for $V(x)=5 \sin(2 \pi (x+\frac 1 4))$ with $d_2=0.5$ and $d_1$ such that \eqref{eq: int_of_H0} holds. }
                \label{pic: m0}
        \end{figure}

        \begin{figure}
                \centering      
                \includegraphics[width=50mm]{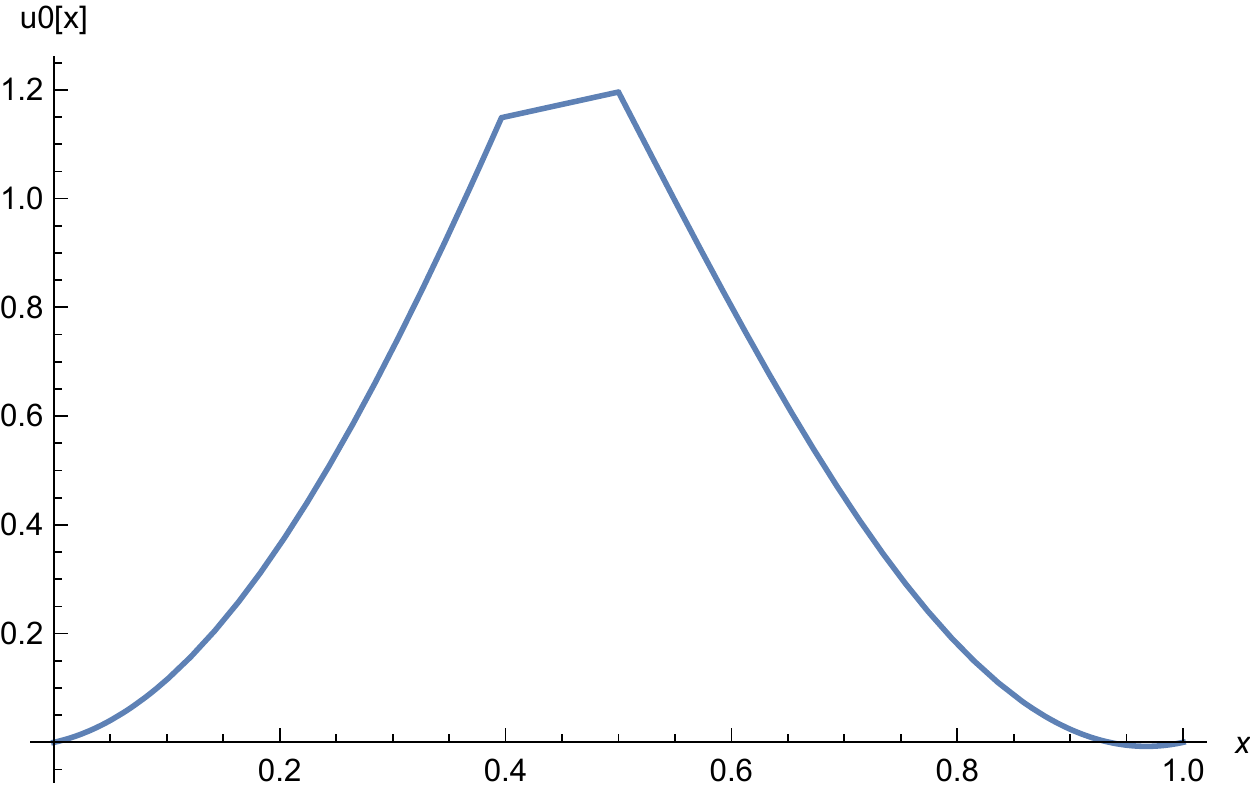}
                \includegraphics[width=50mm]{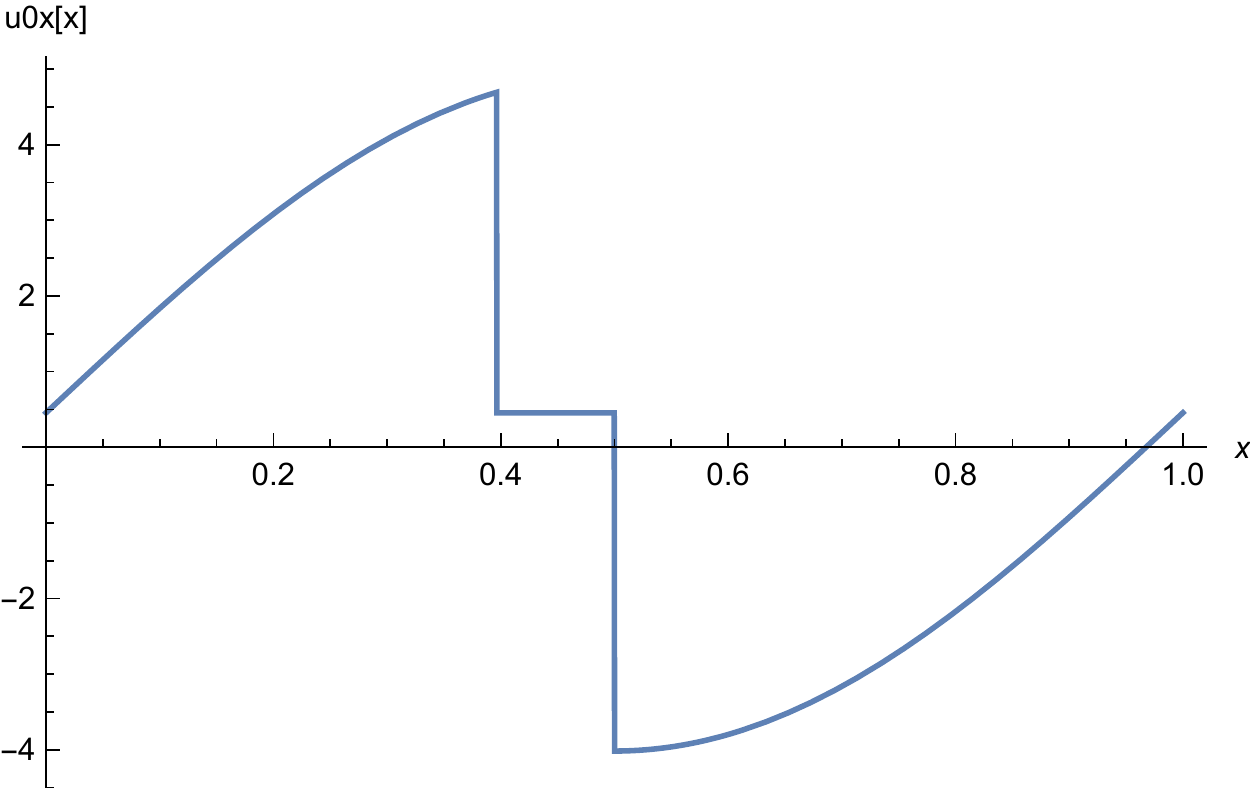}
                \caption{$u_0$ (left) and $(u_0)_x$ (right) as defined in \eqref{eq: m_d1d2} for $V(x)=5 \sin(2 \pi (x+\frac 1 4))$ with $d_2=0.5$ and $d_1$ such that \eqref{eq: int_of_H0} holds.}
                \label{pic: u0}
        \end{figure}

        \begin{remark}\label{rmrk: V_multimax_j=0}
        If $V$ has multiple maxima and Case \textit{ii} in Proposition \ref{prp: j=0} holds,  there is a larger family of solutions. Let $x=x_0 \in (0,1) $ be a point of maximum for $V$. For fixed real numbers,  $d_1<d_2<e_1<e_2$, define
        \begin{equation}\label{eq: m_d1d2e1e2}
        m_0^{d_1,d_2,e_1,e_2}(x)=\begin{cases}
        \Hh_{0}-V(x),\ x\in [d_1,d_2]\cup [e_1,e_2],\\
        0,\ \text{elsewhere}, 
        \end{cases}
        \end{equation}
        and
        \begin{equation}\label{eq: u_d1d2e1e2}
        u_0^{d_1,d_2,e_1,e_2}(x)=\int\limits_{0}^{x} (u_0^{d_1,d_2,e_1,e_2})_x(y) dy,\quad x\in \Tt,
        \end{equation}
        where 
        \begin{equation*}
        (u_{d_1,d_2,e_1,e_2})_x(x)=\begin{cases}
        \sqrt{2(\Hh_{0}-V(x))}-p_0^{d_1,d_2,e_1,e_2},\ x\in [0,d_1) \cup [x_0,e_1),\\
        -\sqrt{2(\Hh_{0}-V(x))}-p_0^{d_1,d_2,e_1,e_2},\ x\in (d_2,x_0] \cup (e_2,1],\\
        -p_0^{d_1,d_2,e_1,e_2},\ \text{elsewhere}.
        \end{cases}
        \end{equation*}
        Note that $(u_{d_1,d_2,e_1,e_2})_x(x)$ is periodic, only has negative jumps, and solves \eqref{eq: currentform_j=0} almost everywhere. Hence, the triplet $(u_0^{d_1,d_2,e_1,e_2},m_0^{d_1,d_2,e_1,e_2},\Hh_{0})$ is a regular solution of \eqref{eq: currentform_j=0} if
        \begin{equation}\label{eq: m}
        \int\limits_{0}^{1} m_{d_1,d_2,e_1,e_2}(x) dx=1
        \end{equation}
        for
        \begin{align*}
        p_0^{d_1,d_2,e_1,e_2}=\int\limits_{[0,d_1) \cup [x_0,e_1)} \sqrt{2(\Hh_{0}-V(x))}dx-\int\limits_{(d_2,x_0] \cup (e_2,1]} \sqrt{2(\Hh_{0}-V(x))}dx.
        \end{align*}
        The equality \eqref{eq: m} is equivalent to
        \begin{equation}\label{eq: int_of_H0_manymax}
        \int\limits_{d_1}^{d_2} (\Hh_{0}-V(x))dx+\int\limits_{e_1}^{e_2} (\Hh_{0}-V(x))dx=1.
        \end{equation}
        Since $\int\limits_{0}^{1} (\Hh_{0}-V(x))dx>1$, we can find infinitely many quadruples $(d_1,d_2,e_1,e_2)$ such that \eqref{eq: int_of_H0_manymax} holds. Hence, we can generate infinitely many solutions of the form \eqref{eq: m_d1d2e1e2}, \eqref{eq: u_d1d2e1e2}.  
        \end{remark}
        
        From Propositions \ref{prp: j>0} and \ref{prp: j=0},  for every regular solution, $m,$ of \eqref{main} in the low-current regime ($j=0$ or Case i in Proposition \ref{prp: j>0}), the smaller $V(x)$ is, the larger  $m(x) $ is.  This is paradoxical because
$V(x)$ represents the
spatial preference of the agents and preferred regions correspond to high values of $V$. 
Thus, areas that are less desirable  have a high populational density. Therefore, it is possible that the most preferred site is empty and agents aggregate at the least preferred site. For example, in \eqref{eq: m_d1d2}, $m$ vanishes near the maximum of $V$ and is supported in the neighborhood of the minimum of $V$, as illustrated in Fig. \ref{pic: m0}. Hence, if agents do not move fast (low current), they prefer staying together rather than being in a better place, see Fig. \ref{fig: casei}. In the high-current regime (Case ii in Proposition \ref{prp: j>0}), the opposite situation occurs: the larger $V(x)$ is, the larger $m(x)$ becomes, see Fig. \ref{fig: caseii}. Therefore, preferred areas have a high population density. Hence, if the level of the current is high enough (we give quantitative estimates in the next section), agents are better off at preferred sites and with more agents.
               Finally, for the intermediate current level (Case iii in Proposition \ref{prp: j>0}), we observe a more complex situation. The solution, $m,$ consists of two parts: $m^-_j(x)$ and $m^+_j(x)$. $m^-_j(x)$ is larger where $V(x)$ is larger and the opposite holds for $m^+_j(x)$. Therefore, in the region where $m$ is $m^-_j$, the most preferred sites are more densely populated. In the region where $m$ is $m^+_j$ the less preferred sites are more densely populated.
This is illustrated in Fig. \ref{fig: caseiii}.

\section{Discontinuous viscosity solutions}
\label{discsec}

In the anti-monotone case considered in the preceding section,  $m$ can be discontinuous. Thus, in addition to regular solutions examined before, we need to consider viscosity solutions in the framework of discontinuous Hamiltonians. In what follows, we recall the main definitions in \cite{Barlesbook}. Given a locally bounded function, $F:\Tt\times\Rr\to \Rr$, we define its lower and upper
semicontinuous envelopes as
\[
F_*(x,q)=\liminf\limits_{(y,r)\to (x,q)} F(y,r), \quad F^*(x,q)=\limsup\limits_{(y,r)\to
(x,q)} F(y,r)
\]
for $(x,q)\in\Tt\times\Rr.$ We say that a locally bounded function,  $u:\Tt\to \Rr$, is a viscosity solution of $F(x, Du)=0$ if, for any smooth function, $\phi:\Tt
\to \Rr,$ we have that
\[F_*(x,\phi_x+p)\leq 0\quad \text{for all}\quad x\in \argmax(u-\phi)
\]
and
\[F^*(x,\phi_x+p)\geq 0\quad \text{for all}\quad x\in \argmin(u-\phi).
\]Let $m:\Tt\to \Rr$, $m\in L^\infty(\Tt)$, and set \[
m_*(x)=\liminf\limits_{y\to x} m(y), \quad m^*(x)=\limsup\limits_{y\to
x} m(y).
\]
Suppose that $V:\Tt\to \Rr$ is continuous. Then, for our setting, we have
\[
 F(x,q)=\frac{q^2}{2}+V(x)+m(x)-\Hh.  
\]
Consequently,
\[F_*(x,q)=\frac{q^2}{2}+V(x)+m_*(x)-\Hh, \quad F^*(x,q)=\frac{q^2}{2}+V(x)+m^*(x)-\Hh.
\]
Here, we look for piecewise smooth solutions of \eqref{main} for $g(m)=-m$ that are not necessarily regular; that is, the condition $\lim\limits_{x\to x^-}u_x(x)\geq \lim\limits_{x\to x^+}u_x(x)$ is not necessarily satisfied. It turns out that there are infinitely many such solutions for all $j\neq 0$ independent of properties of $V$, and the jump direction of $u_x$ is irrelevant. This contrasts with the fact that for $V$ with a single maximum, there exists just one regular solution (Proposition \ref{prp: j>0}).

Thus, we select a current level, $j>0$ ($j<0$ is analogous), and fix arbitrary points $0\leq x_0 <x_1 <\cdots<x_n\leq 1$ and $\Hh \geq \Hh^{cr}_j$. We search for solutions $(u,m,\Hh)$ such that $m$ is continuous on the intervals $(x_i,x_{i+1})$ for $0\leq i\leq n-1$. From the above discussion, we have:
\begin{proposition}
        Assume that $j>0$ and that 
        \begin{itemize}
                \item $m>0$ is continuous on $(x_i,x_{i+1})$ and
                \begin{equation*}
                        \frac{j^2}{2m(x)^2}+m(x)=\Hh-V(x)\ \text{for all}\ x\neq x_i.
                \end{equation*}
                \item $\Hh$ is such that $\int\limits_{\Tt} m(x) dx=1$.
        \end{itemize}
Then, the triplet $(u,m,\Hh)$ solves \eqref{main}, where
        \[u(x)=\int\limits_{0}^{x}\frac{j}{m(y)}dy-px,\quad p=\int\limits_{\Tt}\frac{j}{m(y)}dy.
        \]
\end{proposition}
\begin{proof} We have that $u_x+p=\frac{j}{m}$ a.e.. Thus,  the second equation in \eqref{main} holds in the sense of distributions. Next, we observe that $u$ is differentiable for all $x\neq x_i$ and that the first equation in \eqref{main} is satisfied in the classical sense at those points. Thus, we just need to check the viscosity condition at $x=x_i$.

There are two possible cases:
\begin{enumerate}
        \item $m(x_i^-)>m(x_i^+)$.
        
        In this case,  $m^*(x_i)=m(x_i^-). $  Moreover,   $$u_x(x_i^-)=j/m(x_i^-)-p<j/m(x_i^+)-p=u_x(x_i^+).$$ 
        Hence, there is no smooth function touching $u$ from above; it touches only from below. Therefore, we need to check that, for any $\phi$ touching $u$ from below at $x_i$, we have
        \[\frac{(\phi_x(x_i)+p)^2}{2}+V(x_i)+m(x_i^-)-\Hh \geq 0.  
        \]
Because \eqref{main} is satisfied at $x\neq x_i$ in the classical sense, we have that
        \[\frac{(u_x(x_i^\pm)+p)^2}{2}+V(x_i)+m(x_i^\pm)-\Hh=0.
        \]
        Because
        $\phi$ touches $u$ from below and $j>0$, we have
        \[0<u_x(x_i^-)+p\leq\phi_x(x_i)+p\leq u_x(x_i^+)+p.
        \]
        Hence,
        \begin{align*}&\frac{(\phi_x(x_i)+p)^2}{2}+V(x_i)+m(x_i^-)-\Hh\\
        &\geq \frac{(u_x(x_i^-)+p)^2}{2}+V(x_i)+m(x_i^-)-\Hh = 0.
        \end{align*}
        
        \item $m(x_i-)<m(x_i+)$.
        
        In this case,  $m_*(x_i)=m(x_i^-)$  and $$u_x(x_i^-)=j/m(x_i^-)-p>j/m(x_i^+)-p=u_x(x_i^+).$$ 
        Hence, there is no smooth function touching $u$ from below -- only from above. Therefore, for any $\phi$ touching $u$ from above at $x_i$, we have
        \begin{equation}
        \label{vis_cond}
        \frac{(\phi_x(x_i)+p)^2}{2}+V(x_i)+m(x_i^-)-\Hh \leq 0.
        \end{equation}
 Because \eqref{main} holds in the classical sense for $x\neq x_i,$  we have that
        \[\frac{(u_x(x_i^\pm)+p)^2}{2}+V(x_i)+m(x_i^\pm)-\Hh=0.
        \]
  Because      $\phi$ touches $u$ from above, we have
        $0<u_x(x_i^+)+p\leq\phi_x(x_i)+p\leq u_x(x_i^-)+p.
        $
        Hence, \eqref{vis_cond} holds.
\end{enumerate}
\end{proof}
\begin{remark}
If $g$ is increasing, the construction of piecewise smooth solutions with discontinuous $m$ in the previous proposition fails because $m(x_i^-)=m(x_i^+)$, necessarily. Therefore, the smooth solutions found in Proposition \ref{prp:j<>0g_increasing} are the only possible ones: there are no extra discontinuous solutions as in the case of decreasing $g$. This is yet another consequence of the regularizing effect of an increasing $g$.
\end{remark}

\section{Anti-monotone elliptic mean-field games}\label{sec:amonotoneelliptic}

Now, we consider anti-monotone elliptic MFGs and the corresponding variational problem \eqref{mainef} with $g(m)=-m$. We use the direct method in the calculus of variations to prove the  existence of a minimizer of  the functional
\begin{equation}
\label{Je}
J_\epsilon[m] = \int_\Tt \left( \epsilon^2 \frac{m_x^2}{2m}+\frac{j^2}{2m}
- \frac{m^2}{2} -V(x) m \right) dx.
\end{equation}

\begin{pro}
        For each $j\in\Rr$, there exists a minimizer, $m$, of $J[m]$ in 
        $$
        \mathcal{A} = \left\{ m\in W^{1,2}(\Tt): m> 0 \wedge \int_{\Tt} m =1   \right\}.
        $$
        Moreover, $m,$ solves
        $$
        -\epsilon^2 \left(\frac{m_x}{m}\right)_x - \frac{j^2}{2m^2} -m + \Hh -V(x) =0
        $$
        for some $\Hh\in\Rr$.
\end{pro}

\begin{proof}
        To prove the existence of a positive minimizer, we consider separately the cases $j\neq 0$ and $j=0$.\\
\subparagraph{\bf{ Case 1. $\bf j\neq 0$}}
We take a minimizing sequence, $m_n\in\mathcal{A,}$ and note that there is a constant, $C,$ such that
\[
\int_\Tt  \epsilon^2 \frac{(m_n)_x^2}{2m_n}+\frac{j^2}{2m_n}  dx \leq C + \int_\Tt   \frac{m_n^2}{2}.
\]
Therefore,  we seek to control $\int m^2$ by the integral expression on the left-hand side. \\
For that, we recall the Gagliardo-Nirenberg inequality,
\begin{equation}
\label{GNineq}
\left|\left|w\right|\right|_{L^p} \leq \left|\left|w_x\right|\right|_{L^r}^a \left|\left|w\right|\right|_{L^q}^{1-a}  
\end{equation}
for $0\leq a \leq 1$, with
\[
\frac 1 p = a \left(\frac 1 r - 1\right) + (1-a) \frac 1 q, 
\]
whenever $\int_{\Tt} w=0$. With $p=4$ and $r=q=2$, we obtain $a=\frac 1 4$. Using these values in \eqref{GNineq}, taking into account that  $\int m =1, $ and choosing $w=\sqrt{m}$, we obtain
\[
\int_{\Tt} m_n^2 \leq C+C 
\left( \int_{\Tt}\frac{(m_n)_x^2}{2m_n}\right) ^\frac 1 2.
\]
Thus, using a weighted Cauchy inequality, 
\[
\int_\Tt  \epsilon^2 \frac{(m_n)_x^2}{2m_n}+\frac{j^2}{2m_n}  dx \leq C.
\]
Finally, we argue as in the proof of Proposition \ref{p41} and show the existence of a minimizer. 

\subparagraph{\bf{ Case 2. $\bf j= 0$}}
 Here, we use a fixed point argument as in the proof of Proposition \ref{p41}.  For that,  we rewrite the Euler-Lagrange equation as
\begin{equation}
\label{elee}
-\epsilon^2(\ln m)_{xx} - m -V(x) = -\Hh, 
\end{equation}
and argue as before. 
However, because the functional \eqref{Je} is non-convex, uniqueness may fail.  
\end{proof}

The preceding result does not give a unique minimizer. We note that for large  $j$, the functional  \eqref{Je} behaves like a convex functional. Finally, we note that as $\epsilon\to 0$, numerical evidence suggests that there is no $\Gamma$-convergence to a minimizer, see figure \ref{fig: eliptic2}, where we plot a solution for small $\epsilon$ versus the solution with $\epsilon=0$.

\begin{figure}
        \centering      
        \includegraphics[width=50mm]{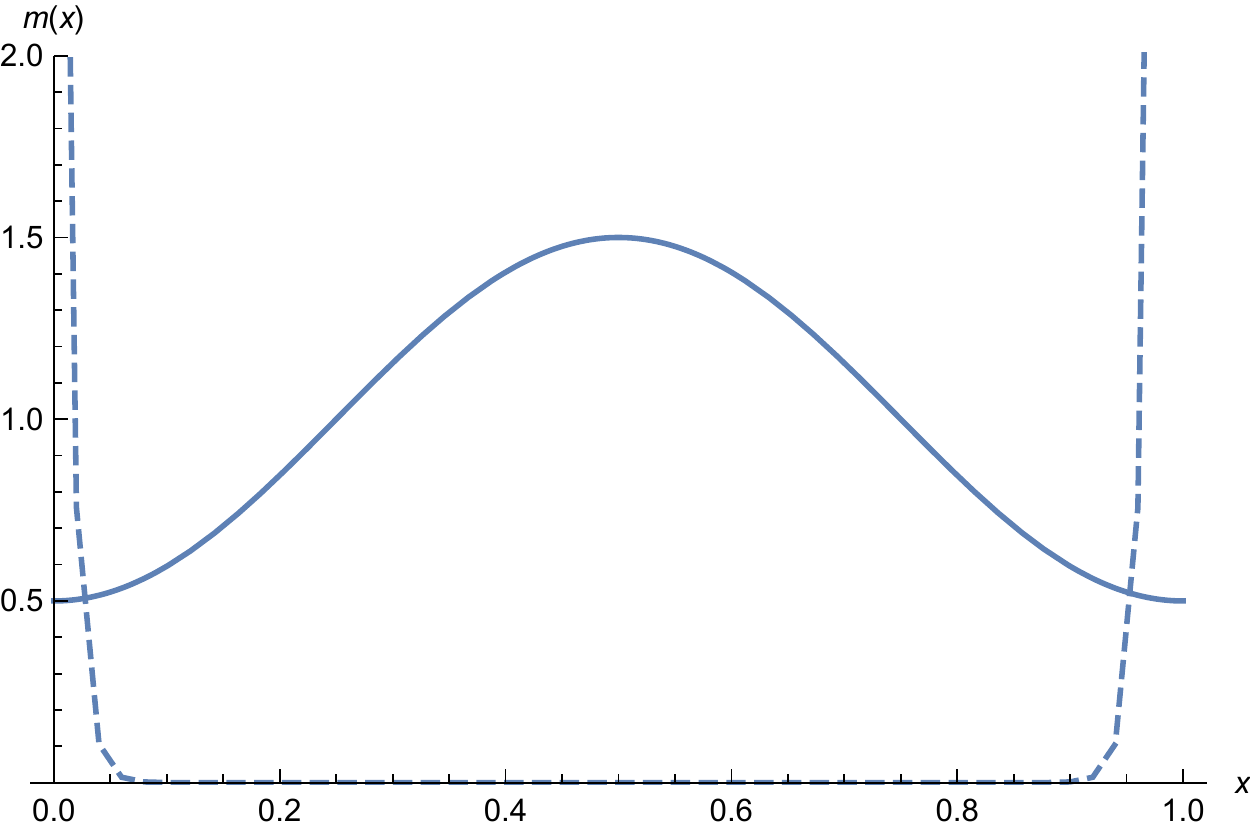}
        \caption{Solution $m$ of \eqref{maine} when $g(m)=-m,\ j=1,\ V(x)=\frac 1 2 \sin(2 \pi (x + 1/4))$ for $\epsilon=0.01$ (dashed) and for $\epsilon=0$ (solid).}
                \label{fig: eliptic2}
        \end{figure}

      \section{Regularity regimes of the current equation for $g(m)=-m$}\label{sec: regmodes}
      
      Now, we analyze the regularity regimes of \eqref{eq: currentform_1}; that is, we determine for which values of $j$ \eqref{eq: currentform_1} has or fails to have smooth solutions. For simplicity, we assume that $0$ is the only point of maximum of $V$. Moreover, as before, we consider the case $j\geq 0$, as the case $j<0$ is analogous.
      
      We begin by proving that $\alpha^+,\alpha^-$, defined in \eqref{eq: a+-}, are monotone. 
      \begin{proposition}\label{prp: a+,-}
        We have that
        \begin{itemize}
                \item[i.] $\alpha^+$ and $\alpha^-$ are increasing on $(0,\infty)$;
                \item[ii.] $\lim\limits_{j\to+\infty}\alpha^+(j)=\lim\limits_{j\to+\infty}\alpha^-(j)=\infty$;
                \item[iii.] $\lim\limits_{j\to 0} \alpha^+(j)=\max\limits_{\Tt}V-\int\limits_{\Tt} V(x) dx, \ \lim\limits_{j\to 0} \alpha^-(j)=0$.
        \end{itemize}
      \end{proposition}
      \begin{proof}
        \textbf{i.} First, we prove that $m^+_j(x)$ and $m^-_j(x)$ (see Section \ref{sec: j>0} for the definition) are increasing in $j$ at every point $x \in \Tt$.
        We fix $x$ and set $h=(\max\limits_{\Tt} V )-V(x)$. If $h=0$, then $m^-_j(x)=j^{2/3}$, which is an increasing function of $j$. Next, for $h>0$ let $t(j)=m^-_j(x)<j^{2/3}$. We have that
        \begin{equation*}
                \frac{j^2}{2t(j)^2}+t(j)-\frac{3}{2}j^{2/3}=h.
        \end{equation*}
        By the implicit function theorem, $t(j)$ is differentiable. Differentiating the previous equation in $j$ gives
        \begin{equation*}
                t'(j)=\frac{j^{-1/3}-\frac{j}{t^2}}{1-\frac{j^2}{t^3}}.
        \end{equation*}
                        
        Because $0<t(j)<j^{2/3},\ t'(j)>0$. Hence, $t(j)$ is increasing. The proof for $m^+_j(x)$ is identical.
        
        \textbf{ii.} By definition, $m^+_j(x) \geq j^{2/3}$. Hence, $\lim\limits_{j\to\infty}\alpha^+(j)=\infty$.
        On the other hand,  for $j$ large enough, we have
        \begin{align*}
                \frac{j^2}{2(j^{2/3}/2)^2}+\frac{j^{2/3}}{2}-\frac{3}{2}j^{2/3}=j^{2/3}&>\max \limits_{\Tt}V- V(x)\\ \nonumber
                &=\frac{j^2}{2(m^-_j(x))^2}+m^-_j(x)-\frac{3}{2}j^{2/3}.
        \end{align*}
        
        Therefore, $m^-_j(x)>j^{2/3}/2 $ and  $\lim\limits_{j\to\infty}\alpha^-(j)=\infty$.
        
        \textbf{iii.} Because $m^-_j(x) \leq j^{2/3}$ for every $x\in \Tt$, $\lim\limits_{j\to 0}m^-_j(x)=0$ for all $x\in \Tt$. Thus, $\lim\limits_{j\to 0}\alpha^-(j)=0$. On the other hand, $m^+_j(x) \geq j^{2/3}$. Thus, $0\leq\frac{j^2}{(m_j^+(x))^2} \leq j^{2/3}$. Therefore,
        \begin{equation*}
                \lim\limits_{j \to 0}   m^+_j(x)=\lim\limits_{j \to 0}\left(\frac{3}{2}j^{2/3}-\frac{j^2}{2(m_j^+(x))^2}+\max\limits_{\Tt}V-V(x)\right)=\max\limits_{\Tt}V-V(x).
        \end{equation*}
        Thus,
        \begin{equation*}
                \lim\limits_{j\to 0} \alpha^+(j)=\max\limits_{\Tt}V-\int\limits_{\Tt} V(x) dx.
        \end{equation*}
      \end{proof}
      Next, we define two numbers that characterize regularity regimes of \eqref{main}:
      \begin{equation}\label{eq: j_lower}
        j_{lower}=\inf \{j>0\ \text{s.t.} \ \alpha^{+}(j)> 1\},
      \end{equation}
      and
      \begin{equation}\label{eq: j_upper}
        j_{upper}=\inf \{j>0\ \text{s.t.} \ \alpha^{-}(j)> 1\}.
      \end{equation}
      \begin{proposition}\label{prp: regmodes}
        Let $j_{lower}$ and $j_{upper}$ be given by \eqref{eq: j_lower} and \eqref{eq: j_upper}. Then
        \begin{itemize}
                \item[i.] $0\leq j_{lower} < j_{upper}<\infty$;
                \item[ii.] for $j\geq j_{upper}$, the system \eqref{main} has smooth solutions;
                \item[iii.] for $j_{lower}<j< j_{upper}$, the system \eqref{main} has only discontinuous solutions;
                \item[iv.] if $j_{lower}>0$, the system \eqref{main} has smooth solutions for $0<j\leq j_{lower}$.
        \end{itemize} 
      \end{proposition}
      \begin{proof}
        The proof is a straightforward application of Propositions \ref{prp: j>0} and \ref{prp: a+,-}.
      \end{proof}
      Finally, we characterize the regularity at $j=0$.
      \begin{proposition}\label{prp: regmode_j=0}
        The system \eqref{eq: currentform_j=0} admits smooth solutions if and only if
        \begin{equation*}
                \alpha^+(0)\leq 1.
        \end{equation*}
      \end{proposition}
      \begin{proof}
        The proof follows from \textit{iii} in Proposition \ref{prp: a+,-} and \textit{i} in Proposition \ref{prp: j=0}.
      \end{proof}
      
      Let $V(x)=A \sin(2\pi(x+1/4))$. In Fig. \ref{fig: a+,-}, we plot $\alpha^+$ and $\alpha^-$ for $A=0.5$ and $A=5$. From Proposition \ref{prp: a+,-},  $\alpha^+(0)=A$. Thus, if $A=0.5,$ we have $\alpha^+(0)<1$ and, for $A=5,$ we have $\alpha^+(0)>1$.
      Therefore, $j_{lower}>0$ for $A=0.5$ and $j_{lower}=0$ for $A=5$. Hence, if $A=0.5,$ \eqref{main} has smooth solutions for a low enough current level  ($j\leq0.218$) or for a high enough current level  ($j\geq1.750$).
      In contrast, if $A=5,$ there are no smooth solutions for low currents,  only for large currents ($j\geq 3.203$).
      
      We end the section with an a priori estimate for the current level for smooth solutions.
      \begin{proposition}[A priori estimate]
        Suppose that $\max\limits_{\Tt} V >1+\int\limits_{\Tt}V(x)dx$ and let $(u,m,\Hh)$ be a smooth solution of \eqref{main} with $m>0$. Then, there exists a constant, $c(V)>0$, such that
        \begin{equation}\label{eq: aprioribound}
                \inf \limits_{\Tt} m (u_x+p) \geq c(V).
        \end{equation}
      \end{proposition}
      \begin{proof}
        From Proposition \ref{prp: a+,-}, we have that $\alpha^+(0)>1$. Thus, by Proposition \ref{prp: regmode_j=0}, \eqref{main} does not have smooth solutions for $j=0$. Additionally, $j_{lower}=0$. Next, take $c(V)=j_{upper}$ and by Proposition \ref{prp: regmodes}, we conclude \eqref{eq: aprioribound}.
      \end{proof}
      The previous Proposition shows that if the potential, $V,$ has a large oscillation (this happens in the example for $A=5$, Fig. \ref{fig: a+,-}), then only high current solutions are smooth.

        \begin{figure}
        \centering      
        \includegraphics[width=50mm]{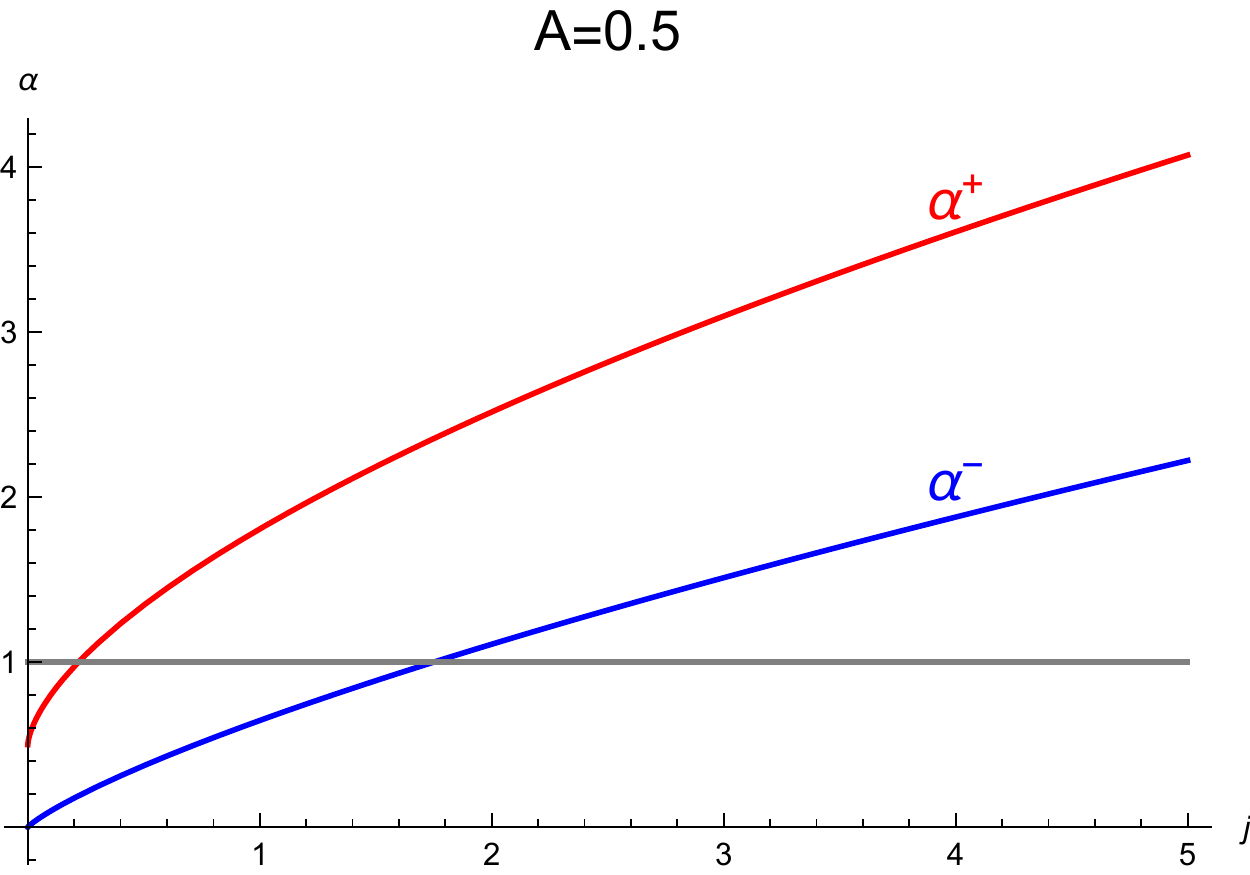}
        \includegraphics[width=50mm]{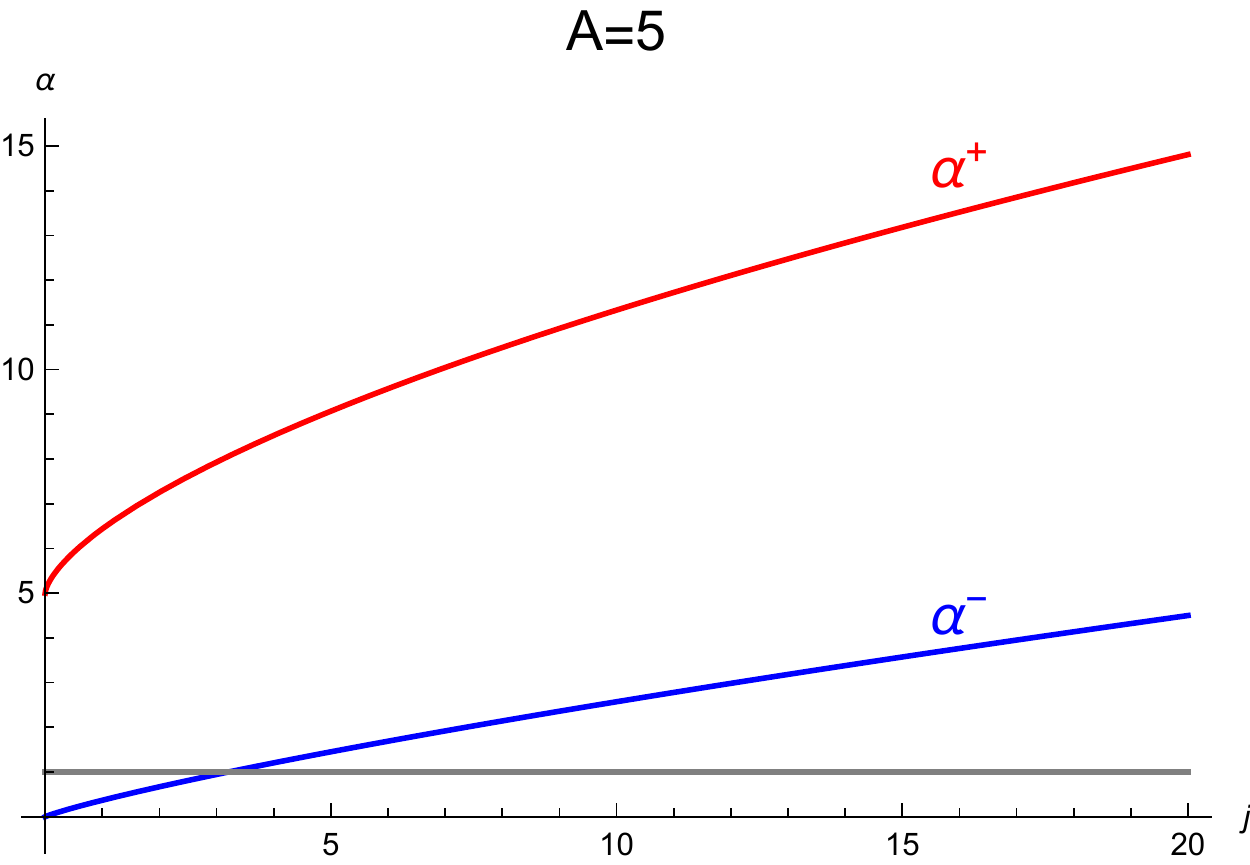}
        \caption{$\alpha^+$ and $\alpha^-$ for $V(x)=A \sin(2 \pi (x + 1/4))$. $j_{lower}=0.218,\ j_{upper}=1.750\ (A=0.5);\ j_{lower}=0,\ j_{upper}=3.203\ (A=5)$.}
        \label{fig: a+,-}
        \end{figure}

\section{Asymptotic behavior of solutions as $j \to 0$ and $j \to +\infty$} \label{sec: asymptotics}
        
In Section \ref{sec: j>0}, we studied regular solutions of \eqref{main} with a current level $j>0$. Here, we continue the analysis of the decreasing nonlinearity, $g(m)=-m,$ and examine the asymptotic behavior of regular solutions as $j \to 0$ and $j \to \infty$.
           
        As before, we assume that  $V$ has a single maximum at $0$.
        First, we address the case $j\to\infty$.        
        \begin{proposition}\label{prp: j_to_infty}
        For $j>0$, let $(u_j,m_j,\Hh_j)$ solve \eqref{eq: currentform_1}. We have that
        \begin{itemize}
        \item[i.] $\lim\limits_{j \to \infty} \Hh_j=\infty$;
        \item[ii.] For $x \in \Tt$, $\lim\limits_{j \to \infty} m_j(x)=1$, $\lim\limits_{j \to \infty} u_j(x)=0$, and $\lim\limits_{j\to \infty}p_j=\infty$.
        \end{itemize}
        \end{proposition}
        
        \begin{proof}
        i. According to \eqref{eq: Hboundbybelow}, we have that $\Hh_j \geq \frac{3j^{2/3}}{2}+\max \limits_{\Tt} V$. Thus,  $\lim\limits_{j \to \infty} \Hh_j=\infty$.
        
        ii. For $j\geq j_{upper}$, solutions of \eqref{eq: currentform_1}
        are given by \eqref{eq: sols_ii}.
        Hence, $m_j$ consists only of the $m^-$ branch. Thus, $m_j(x) \leq j^{2/3}$, which yields $\frac{j^2}{m_j(x)^2}\geq m_j(x)$. Therefore,
        \(
        \frac{j^2}{2m_j(x)^2}+m_j(x)\leq \frac{3j^2}{2m_j(x)^2}.
        \)
        Consequently, using this inequality in \eqref{eq: currentform_1}, we get 
        \begin{equation}\label{eq: 6}
        \frac{j}{\sqrt{2(\Hh_j-V(x))}} \leq m_j(x)\leq \frac{\sqrt{3}j}{\sqrt{2(\Hh_j-V(x))}}.
        \end{equation}
        Integrating the previous inequality and taking into account that $\int\limits_{\Tt}m_j(x)=1$, we get
        \begin{equation}\label{eq: 7}
        \int_{\Tt}\frac{j}{\sqrt{2(\Hh_j-V(x))}} dx \leq 1\leq \int_{\Tt}\frac{\sqrt{3}j}{\sqrt{2(\Hh_j-V(x))}}dx.
        \end{equation}
        Because $\Hh_j$ converges to $\infty,$ and,
        for every $x,y \in \Tt$, where $V$ is bounded, we have that
        \[\lim\limits_{j\to \infty}\frac{\sqrt{2(\Hh_j-V(y))}}{\sqrt{2(\Hh_j-V(x))}}=1.
        \]
        Hence, for large enough $j$, we have
        \begin{equation}\label{eq: 8}
                \sqrt{2(\Hh_j-V(x))} \leq 2 \sqrt{2(\Hh_j-V(y))},\quad x,y \in \Tt.
        \end{equation} 
        Let $\bar{x}$ be such that
        \begin{equation*}
                \frac{j}{\sqrt{2(\Hh_j-V(\bar{x}))}}=\int_{\Tt}\frac{j}{\sqrt{2(\Hh_j-V(x))}} dx.
        \end{equation*}
        Then, by \eqref{eq: 6}, \eqref{eq: 8}, and \eqref{eq: 7}, we get
        \begin{align*}
                m_j(x)\leq \frac{\sqrt{3}j}{\sqrt{2(\Hh_j-V(x))}}\leq \frac{2\sqrt{3}j}{\sqrt{2(\Hh_j-V(\bar{x}))}}\leq 2\sqrt{3}.
        \end{align*}
        Similarly, we have
        \begin{align*}
                m_j(x)\geq \frac{j}{\sqrt{2(\Hh_j-V(x))}}\geq \frac{j}{2\sqrt{2(\Hh_j-V(\bar{x}))}}\leq \frac{1}{2\sqrt{3}}.
        \end{align*}
        Furthermore, we have that
        \begin{equation}
        \label{lll3}
        \frac{j^2}{2m_j^2(x)}=\frac{\Hh_j-V(x)}{1+\frac{2m_j^3(x)}{j^2}}.\end{equation}
        Thus,
        \begin{equation}\label{eq: 2}
                m_j(x)=\frac{1}{\sqrt{1+\frac{2m_j^3(x)}{j^2}}}\frac{j}{\sqrt{2(\Hh_j-V(x))}}.
        \end{equation}             
        Finally, because $m_j$ is bounded and its integral is 1, we get from \eqref{eq: 2} that
        \begin{equation}\label{eq: 3}
                \lim\limits_{j \to \infty} \frac{j}{\sqrt{2(\Hh_j-V(x))}}=1
        \end{equation}
        for all $x \in \Tt$. The preceding limit implies that $\lim\limits_{j \to \infty}m_j(x)=1$ for all $x \in \Tt$.
        In fact, \eqref{eq: 3} gives precise asymptotics of $\Hh_j$, namely
        \begin{equation}\label{eq: H_j_asymptotics}
                \lim\limits_{j \to \infty} \frac{2\Hh_j}{j^2}=1.
        \end{equation}
        
        Now, we compute the limit of $u_j(x)$. We have that
        \((u_j)_x=\frac{j}{m_j(x)}-p_j,
        \)
        where $p_j=\int\limits_{\Tt}\frac{j}{m_j(y)}dy$. From \eqref{eq: currentform_1}, we have that
        \(\frac{j}{m_j(x)}=\sqrt{2(\Hh_j-V(x)-m_j(x))}\).
        Therefore, using \eqref{lll3},
        \begin{align*}&\left|\frac{j}{m_j(x)}-\frac{j}{m_j(y)}\right|=\left|\sqrt{2(\Hh_j-V(x)-m_j(x))}-\sqrt{2(\Hh_j-V(y)-m_j(y))}\right|\\
                &\leq\frac{|m_j(x)-m_j(y)|+|V(x)-V(y)|}{\sqrt{2(\Hh_j-\min\limits_{\Tt}V)}}
                \leq \frac{2\sqrt{3}+\text{osc}V}{\sqrt{2(\Hh_j-\min\limits_{\Tt}V)}}\to 0,
        \end{align*}
        as $j\to\infty$.
        Hence,
        \begin{align*}
                |(u_j)_x|&=\left|\frac{j}{m_j(x)}-p_j\right|=\left|\int\limits_{\Tt}\left(\frac{j}{m_j(x)}-\frac{j}{m_j(y)}\right)dy\right|\\
                &\leq \int\limits_{\Tt}\left|\frac{j}{m_j(x)}-\frac{j}{m_j(y)}\right|dy\to 0,
        \end{align*}
        when $j \to \infty$. Consequently, $\lim\limits_{j \to \infty}u_j(x)=\lim\limits_{j \to \infty} \int\limits_{0}^{x} (u_j)_x(y)dy=0$.
      \end{proof}
      Next, we study the behavior of solutions as $j \to 0$.
      \begin{proposition}\label{prp: j_to_0}
        We have that
        \begin{itemize}
                \item[i.] $\lim\limits_{j \to 0} \Hh_j=\max \left(\max\limits_{\Tt} V,1+\int\limits_{\Tt} V\right)=\Hh_0$;
                \item[ii.] if $1+\int\limits_{\Tt} V > \max \limits_{\Tt} V$, then
                \begin{equation*}
                        \lim\limits_{j \to 0} m_j(x)=1+\int\limits_{\Tt} V-V(x),\quad  \lim\limits_{j \to 0} u_j(x)=0,\quad \text{and}\  \lim\limits_{j\to 0}p_j=0
                \end{equation*}
                for all $x \in \Tt$;
                \item[iii.] if $1+\int\limits_{\Tt} V \leq  \max \limits_{\Tt} V$, then
                \begin{equation*}
                        \lim\limits_{j \to 0} m_j(x)=m_{d,1}(x),\ \  \ \lim\limits_{j \to 0} u_j(x)=u_{d,1}(x),\ \text{and}\ \lim\limits_{j\to 0}p_j=\int\limits_{0}^{d} \sqrt{2(\max \limits_{\Tt} V-V(x))}dx
                \end{equation*}
                for all $x \in \Tt$, where $m_{d,1}$ and $u_{d,1}$ are given by \eqref{eq: m_d1d2} and \eqref{eq: u_d1d2}.
        \end{itemize}
      \end{proposition}
      \begin{proof}
        i. There are two possible cases: $j_{lower}>0$ and $j_{lower}=0$.
        If $j_{lower}=0$, then $\alpha^+(j)>1$ for all $j>0$ and $\alpha^{-}(j)<1$ for small enough $j$. Hence, by the results in Section \ref{sec: j>0}, we have that $\Hh_j=\Hh_j^{cr}=\frac{3}{2}j^{2/3}+\max V$. Thus,  $\lim\limits_{j \to 0}\Hh_j=\max V$. On the other hand,  $j_{lower}=0$ means that $\lim\limits_{j\to 0}\alpha^+(j)\geq 1$. Consequently, by Proposition \eqref{prp: a+,-},  $\max\limits_{\Tt} V-\int \limits_{\Tt}V\geq 1$. Thus,  $\max \left(\max\limits_{\Tt} V,1+\int\limits_{\Tt} V\right)=\max\limits_{\Tt} V=\lim\limits_{j\to 0}\Hh_j$.
        
        If $j_{lower}>0$, then $\alpha^+(j)<1$ for $j<j_{lower}$ and solutions $(m_j,u_j,\Hh_j)$ are given by \eqref{eq: sols_i}. Hence,  $m_j(x)\geq j^{2/3}$ and
        \begin{equation}\label{eq: 1}
                0<\frac{j^2}{2m_j(x)^2}\leq \frac{j^{2/3}}{2}.
        \end{equation}
        Therefore,
        \begin{equation*}
                \lim\limits_{j\to 0}\Hh_j=\lim\limits_{j\to 0}\left(\int\limits_{\Tt} V+\int\limits_{\Tt}m_j+\int\limits_{\Tt}\frac{j^2}{2m_j(x)^2}\right)= \int\limits_{\Tt} V+ 1.
        \end{equation*}
        But $\max\limits_{\Tt} V-\int\limits_{\Tt}V=\lim\limits_{j\to 0} \alpha^+(j)<1$, so $\max \left(\max\limits_{\Tt} V,1+\int\limits_{\Tt} V\right)=1+\int\limits_{\Tt} V=\lim\limits_{j\to 0}\Hh_j$.
        
        ii. Since $\lim\limits_{j\to 0}\alpha^+(j)=\max V -\int V$, we have that the condition $1+\int\limits_{\Tt} V> \max \limits_{\Tt} V$ is equivalent to the condition $j_{lower}>0$. In this case, we have that $\lim\limits_{j \to 0}\Hh_j=1+\int \limits_{\Tt}V$. Therefore, from \eqref{eq: 1}, we have that
        \begin{equation*}
                \lim\limits_{j \to 0} m_j(x)=\lim\limits_{j \to 0} \left(\Hh_j-V(x)-\frac{j^2}{2m_j^2(x)}\right)=1+\int\limits_{\Tt} V-V(x).
        \end{equation*}
        Furthermore,
        \begin{equation*}
                \lim\limits_{j \to 0} u_j(x)=\lim\limits_{j \to 0}\int\limits_{0}^x \left(\frac{j}{m_j(y)}-\int\limits_{\Tt}\frac{j}{m_j(z)}dz\right)dy=0.
        \end{equation*}
        
        iii. The inequality $1+\int\limits_{\Tt} V \leq  \max \limits_{\Tt} V$ is equivalent to $j_{lower}=0$. Hence, for $0<j<j_{upper}$ solutions are given by \eqref{eq: sols_iii}.
        
        Because $0<m_j^-(x)\leq j^{2/3}$,  $\lim\limits_{j \to 0} m^-_j(x)=0$. Furthermore,   $m_j^+(x)\geq j^{2/3}$.  Thus,
        \[\lim\limits_{j\to 0}\frac{j^2}{2(m_j^+(x))^2}=0.
        \]
        Therefore,
        \[\lim\limits_{j\to 0}m_j^+(x)=\lim\limits_{j \to 0}\left(\Hh_j-V(x)-\frac{j^2}{2(m_j^+(x))^2}\right)=\max V-V(x).
        \]
        Suppose that the jump points, $d_j,$ of $m_j(x)$ (see \eqref{eq: sols_iii}) converge to some $d \in [0,1]$ through a subsequence. Then, through that subsequence
        \(\lim\limits_{j\to 0} m_j(x)=m_0^{d,1}(x),\)
        where $m_0^{d,1}$ is defined in \eqref{eq: m_d1d2}. Hence,
        \[1=\int\limits_{\Tt} m_0^{d,1}(x)dx=\int\limits_d^1 \left(\max V- V(x)\right)dx.
        \]
        Because $V$ has a single maximum, $d$ is defined uniquely by the previous equation. Hence, $\lim\limits_{j\to 0}d_j=d$ and $\lim\limits_{j\to 0} m_j(x)=m_0^{d,1}(x)$, globally (not only through some subsequence). Consequently,
        \[\lim\limits_{j \to 0} u_j(x)=u_0^{d,1}(x),\ \lim\limits_{j\to 0}p_j=p_0^{d,1}=\int\limits_{0}^{d} \sqrt{2(\max \limits_{\Tt} V-V(x))}dx,
        \]
        where $d$ is such that $\int\limits_d^1 \left(\max V- V(x)\right)=1$.
      \end{proof}
      From Proposition \ref{prp: j_to_0}, we see that we recover only part of the solutions for $j=0$ as limits of solutions for $j>0$. If we consider the solutions of \eqref{eq: currentform_1} for which $m$ takes negative values, we recover all solutions described in Section \ref{sec: j=0}.
      Indeed, the first equation in \eqref{eq: currentform_1} is a cubic equation in $m(x)$. Thus,  for every $x \in \Tt,$ there are three solutions: two positive and one negative. Because we are interested in the MFG interpretation of \eqref{eq: currentform_1}, we neglect solutions with negative $m$. However, we can construct solutions for \eqref{eq: currentform_1} without the constraint $m>0$. As $j$ converges to 0, the negative parts of these solutions converge to 0, and, in the limit, we obtain all non-negative solutions of \eqref{eq: currentform_j=0} given in Proposition \ref{prp: j=0}.

      \section{Properties of $\Hh_j$}\label{sec: j}
      
      In this section, we study various properties of the effective Hamiltonian, $\Hh_j$, as a function of $j$. In the following proposition, we collect several properties of $\Hh_j$. 
      \begin{proposition}\label{prp: H_j} We have that
        \begin{itemize}
                \item[i.] For every $j \in \Rr,$ there exists a unique number, $\Hh_j$, such that \eqref{main} has solutions with a current level $j$;
                \item[ii.] $\Hh_j$ is even; that is, $\Hh_j=\Hh_{-j}$;
                \item[iii.] $\Hh_j$ is continuous;
                \item[iv.] $\Hh_j$ increasing on $(0,\infty)$ and decreasing on $(-\infty,0)$;
                \item[v.] $\min\limits_{j \in \Rr}\Hh_j=\Hh_0=\max \left(\max\limits_{\Tt} V, 1+\int\limits_{\Tt}V(x)dx\right)$;
                \item[vi.] $\lim\limits_{|j|\to \infty} \frac{\Hh_j}{j^2/2}=1$.
        \end{itemize}
      \end{proposition}
      \begin{proof}\begin{itemize}
                \item[i.] This follows from Propositions \ref{prp: j>0} and \ref{prp: j=0}.
                \item[ii.] This follows from the fact that $j\mapsto j^2/2t+t$ is an even function for all $t>0$.
                \item[iii.]  Continuity of $\Hh_j$ follows from the continuity of the mapping $(j,t) \mapsto j^2/2t^2+t$ for $j,t>0$. 
                \item[iv.] Since $\Hh_j$ is even, it suffices to show that it is increasing on $(0,+\infty)$.
                First, we show that $\Hh_j$ is increasing on $(j_{upper},\infty)$. For that, we fix $j_0 >j_{upper}$. We have that $\Hh_{j_0}\geq \Hh_{j_0}^{cr}$. Hence, for any $j_{upper}<j<j_0$ we have that $\Hh_{j_0}\geq \Hh_{j_0}^{cr}>\Hh_{j}^{cr}$. Therefore, the function, $\tilde{m}_j$, determined by
                \begin{equation}\label{eq: mtilde_j}
                        \begin{cases}
                                \frac{j^2}{2(\tilde{m}_j(x))^2}+\tilde{m}_j(x)=\Hh_{j_0}-V(x),\\
                                \tilde{m}_j(x) \leq j^{2/3},
                        \end{cases}
                \end{equation}
                is well defined for all $j_{upper}<j<j_0$. Next, we show that the mapping \[j \mapsto \tilde{m}_j(x)
                \]
                is increasing in $(j_{upper},j_0)$ for all $x \in \Tt$. Indeed, fix $x\in \Tt$ and differentiate \eqref{eq: mtilde_j} in $j$ to obtain
                \[\frac{d\tilde{m}_j(x)}{dj}=-\frac{j\tilde{m}_j(x)}{j^2-\tilde{m}_j(x)^3}<0.
                \]
                Hence, $\tilde{m}_j(x)<m_{j_0}(x),\ x\in \Tt$. Accordingly,
                \[\int\limits_{\Tt} \tilde{m}_j(x)dx<\int\limits_{\Tt} m_{j_0}(x)dx=1.
                \]
                Finally, the previous inequality implies $\Hh_j<\Hh_{j_0}$.
                
                The monotonicity of $\Hh_j$ on $(0,j_{lower})$ (in the case $j_{lower}>0$) can be proven analogously.
                
                Next, for $j_{lower}<j<j_{upper}$, we have that $\Hh_j=\Hh_j^{cr}=\frac{3}{2}j^{2/3}+\max\limits_{\Tt} V$.  $\Hh_j$ is thus evidently monotone.
                \item[v.] This follows from the previous properties of $\Hh_j$ and Proposition \ref{prp: j_to_0}.
                \item[vi.] We have proven this in \eqref{eq: H_j_asymptotics}.
        \end{itemize}
      \end{proof}
      
      In Fig. \ref{pic: Hj} we plot $\Hh_j$ as a function of $j$ for $V(x)=\frac{1}{2}\sin(2\pi(x+\frac{1}{4}))$.

      \begin{figure}
        \includegraphics[width=50mm]{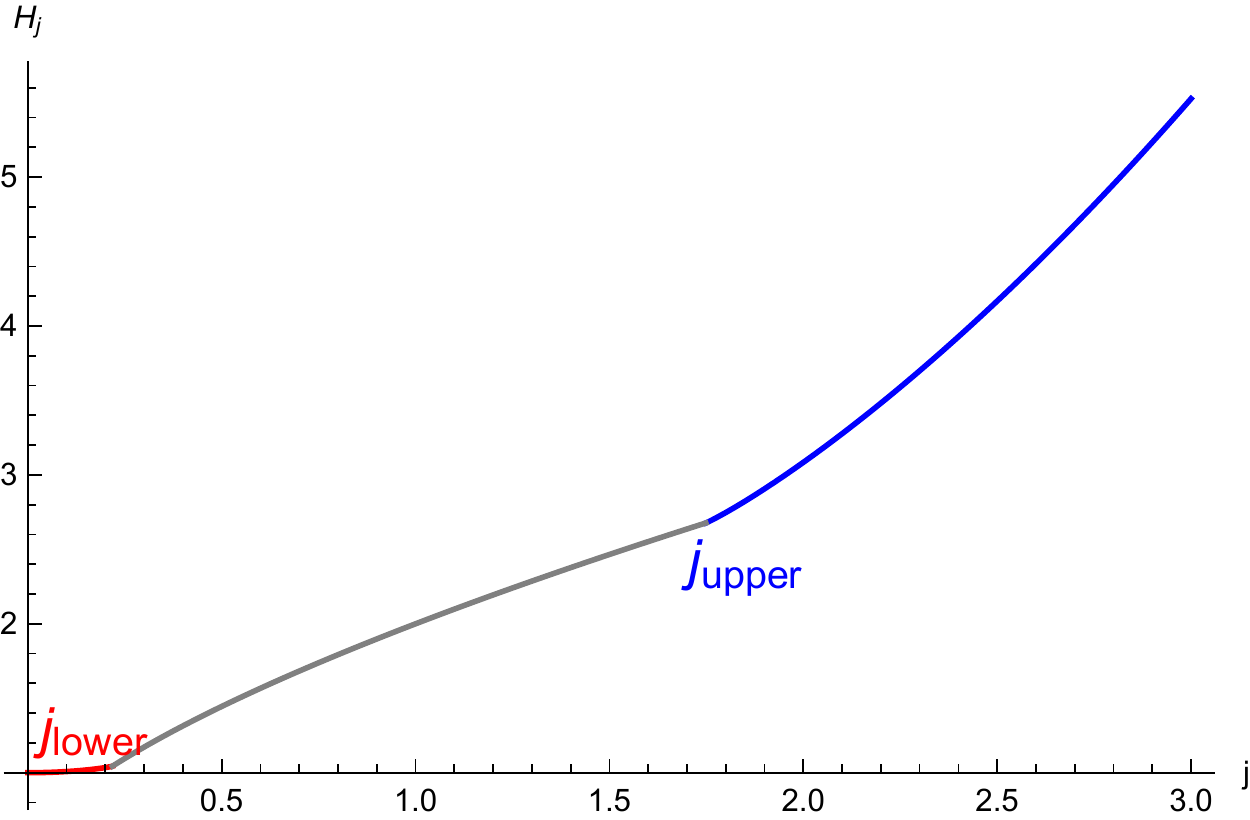}
        \caption{$\bar{H}_j$ for $V(x)=\frac 1 2 \sin(2 \pi (x + \frac 1 4))$.}
        \label{pic: Hj}
      \end{figure}

      \begin{figure}
        \includegraphics[width=50mm]{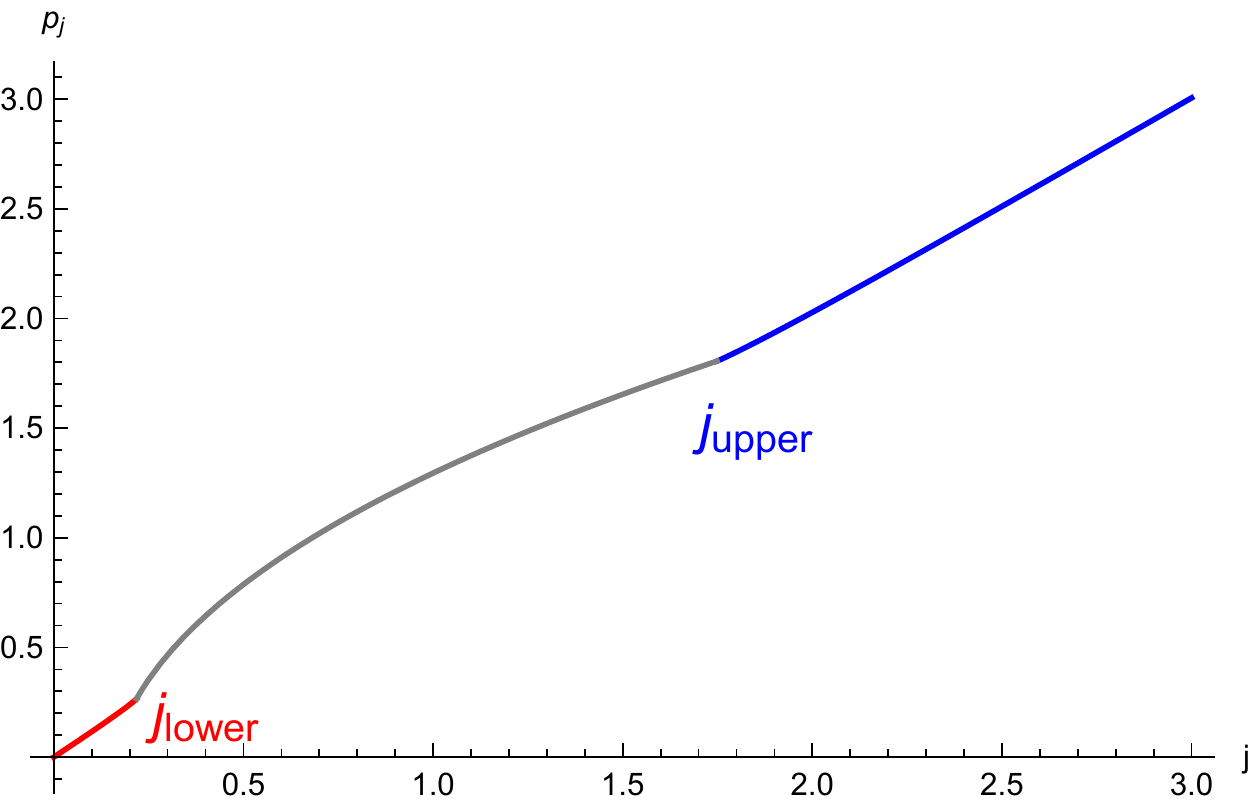}
        \caption{$p_j$ for $V(x)=\frac 1 2 \sin(2 \pi (x + \frac 1 4))$.}
        \label{pic: pj}
      \end{figure}
      
      \section{Analysis in terms of $p$}\label{sec: p}     
      
      Now, we analyze \eqref{main} in terms of the variable $p$. If $g(m)$\ is  increasing,
      for every $p\in \Rr,$ there exists a unique number, $\Hh(p)$, for which \eqref{main} has a solution. This solution is unique if $m>0$ (see, e.g.,  \cite{ll3}).        
      Here, we show that, if $g(m)$ is not increasing, there may be different values  of $\Hh(p)$ for which \eqref{main} has a regular solution. The uniqueness of $\Hh$ depends both on the monotonicity of $g$\ and on the properties of $V$. For example, if $g(m)=-m,$  $\Hh$ is uniquely determined by $p$ if and only if $V$ has a single maximum. Moreover, our prior characterization of regular solutions of \eqref{main} implies that, for $V$ with a single maximum point, \eqref{main} admits a unique regular solution for every $p\in\Rr$. 
      
      We start with an auxiliary lemma.
      \begin{lemma}\label{lma: p-j}
        Let $x=0$ be the single maximum point of $V$.
        \begin{itemize}
                \item[i.] For every $j\neq 0,$ there exists a unique number, $p_{j}$, such that \eqref{main} has a regular solution. Furthermore, the map $j\mapsto p_{j}$ is increasing on $(0,\infty)$ and $(-\infty,0)$.
                \item[ii.] If $1+\int\limits_{\Tt}V(x)dx \geq \max\limits_{\Tt} V$, then $p_0=0$ is the unique number for which  \eqref{main} has a regular solution with $j=0.$ Moreover,  $\lim\limits_{j\to 0}p_{j}=0$.
                \item[iii.] If $1+\int\limits_{\Tt}V(x)dx < \max\limits_{\Tt} V$, then
                \begin{equation*}
                        p_j>\int\limits_{0}^{d_1} \sqrt{2(\max \limits_{\Tt} V-V(x))}dx,\quad j>0
                \end{equation*}
                and
                \begin{equation*}
                        p_j<-\int\limits_{d_2}^{1} \sqrt{2(\max \limits_{\Tt} V-V(x))}dx,\quad j<0,
                \end{equation*}
                where $d_1,d_2 \in (0,1)$ are such that
                \[\int\limits_{0}^{d_1} (\max \limits_{\Tt} V-V(x))dx=\int\limits_{d_2}^{1} (\max \limits_{\Tt} V-V(x))dx=1.
                \]
                Consequently, \eqref{main} has a regular solution for $j=0$ if and only if
                \begin{equation}\label{eq: pbound}
                        -\int\limits_{d_2}^{1} \sqrt{2(\max \limits_{\Tt} V-V(x))}dx\leq p\leq \int\limits_{0}^{d_1} \sqrt{2(\max \limits_{\Tt} V-V(x))}dx.
                \end{equation}
        \end{itemize}
      \end{lemma}
      \begin{proof}
        i. According to Proposition \ref{prp: j>0},  for every $j>0,$ there exists a unique number, $p_j$, such that \eqref{main} has a regular solution with a current level $j$. Let $(u_j,m_j,\Hh_j)$ be the solution of \eqref{main} given by \eqref{eq: sols_i}, \eqref{eq: sols_ii} or \eqref{eq: sols_iii}. Because
        \[p_j=\int\limits_{\Tt}\frac{j}{m_j(y)}dy, 
        \]
        to prove that  $p_j$ is increasing it suffices to show that 
        $j\mapsto\frac{j}{m_j(x)}$ is increasing for all $x\in\Tt$.                  
        First, we prove the monotonicity for $j_{lower}<j<j_{upper}$. Let $n_j(x):=\frac{j}{m_j(x)}$. We have that
        \begin{equation}\label{eq: n_jeq}
                \begin{cases}
                        \frac{n_j^2(x)}{2}+\frac{j}{n_j(x)}=\Hh_j-V(x),\quad x\in \Tt;\\
                        n_j(x)=\frac{j}{m^{-}_{j}(x)}\chi_{[0,d_j)}+\frac{j}{m^{+}_{j}(x)}\chi_{[d_j,1)}.
                \end{cases}
        \end{equation}
        Because the maps $j\mapsto m_j^+(x)$ and $j\mapsto m_j^-(x)$ are increasing for all $x \in \Tt$,  the map $j\mapsto d_j$ is also increasing.
        Assume that  $j$ is such that  $d_j\neq x$. We differentiate in $j$ the first equation in \eqref{eq: n_jeq} and take into account that  $\Hh_j=\frac{3}{2}j^{2/3}+\max\limits_{\Tt}
        V$ for $j_{lower}<j<j_{upper}$ to get
        \[\frac{dn_j(x)}{dj}=\frac{\Hh'_j-\frac{1}{n_j(x)}}{n_j(x)-\frac{j}{n_j^2(x)}}=\frac{\frac{1}{j^{1/3}}-\frac{1}{n_j(x)}}{n_j(x)-\frac{j}{n_j^2(x)}}.
        \]
        Let $j^x$ be such that $x=d_{j^x}$. For $j>j^x,$ we have $d_j>x$. Thus,  $n_j(x)=j/m_j^-(x)>j^{1/3}$, which implies $\frac{dn_j(x)}{dj}>0$. Similarly, for $j<j^x,$ we have $d_j<x$. Therefore, $n_j(x)=j/m_j^+(x)<j^{1/3}$, which implies $\frac{dn_j(x)}{dj}>0$.
        
        Next, we analyze the behavior of $n_j$\ at $j^x$. For $j>j^x,$  $n_j(x)=\frac{j}{m_j^-(x)},$ and, for $j<j^x,$  $n_j(x)=\frac{j}{m_j^+(x)}$. Thus, $n_j(x)$ takes a positive jump, $\frac{j}{m_j^-(x)}-\frac{j}{m_j^+(x)}>0,$ at $j=j^x$.
        Therefore,  $j\mapsto n_j(x)$ has positive derivatives whenever $j\neq j^x$ and a positive jump at $j=j^x$. It is thus increasing for  $j_{lower}<j<j_{upper}$.
        
        Next, we show that $j\mapsto n_j(x)$ is increasing on $(j_{upper},\infty)$. As before, we have
        \[\frac{dn_j(x)}{dj}=\frac{\Hh'_j-\frac{1}{n_j(x)}}{n_j(x)-\frac{j}{n_j^2(x)}}.
        \]
        Because $m_j(x)<j^{2/3}$, we have $n_j(x)>j^{1/3}$. Therefore, if $\Hh'_j\geq 1/n_j(x),$ the map $j\mapsto n_j(x)$ is increasing.
        
        Fix $j_0$ and, for $j>j_0,$ consider $\tilde{H}_j:=\Hh_{j_0}+(j-j_0)\frac{\min\limits_{\Tt} m_{j_0}(x)}{j_0}$. Define
        \begin{equation*}
                \begin{cases}
                        \frac{j^2}{2\tilde{m}_j(x)^2}+\tilde{m}_j(x)=\tilde{H}_j-V(x),\quad x\in \Tt;\\
                        \tilde{m}_j(x)\leq j^{2/3}.
                \end{cases}
        \end{equation*}
        Note that $\tilde{H}_{j_0}=\Hh_{j_0}$ and $\tilde{m}_{j_0}=m_{j_0}$. Now, we compute the derivative of the map $j \to \tilde{m}_j(x)$ at $j=j_0$. Because $m_{j_0}<j_0^{2/3}$, we have
        \begin{equation*}
                \frac{d\tilde{m}_j(x)}{dj}\Bigg|_{j=j_0}=\frac{\tilde{H}'_j-\frac{j}{\tilde{m}_j^2}}{1-\frac{j^2}{\tilde{m}_j^3}}\Bigg|_{j=j_0}=\frac{\frac{\min\limits_{\Tt} m_{j_0}}{j_0}-\frac{j_0}{m_{j_0}^2}}{1-\frac{j_0^2}{m_{j_0}^3}}>0. 
        \end{equation*}
        Thus,                 \begin{equation*}
                \frac{d}{dj}\int\limits_{\Tt}\tilde{m}_j(x)dx \Bigg|_{j=j_0}=\int\limits_{\Tt}\frac{d\tilde{m}_j(x)}{dj}dx>0.
        \end{equation*}
        Hence, for small $j>j_0$ close to $j_0,$ we get
        \[\int\limits_{\Tt}\tilde{m}_j(x)dx>1. 
        \]
        Consequently, for those values of the current,  we have that $\Hh_j\geq \tilde{H}_j$. Hence,
        \[\Hh'_{j_0}\geq \tilde{H}'_{j_0}=\frac{\min\limits_{\Tt} m_{j_0}}{j_0}=\max \limits_{\Tt} \frac{1}{n_{j_0}},
        \]
        which completes the monotonicity proof for $j>j_{upper}$.
        The monotonicity for $j<j_{lower}$ is similar.
        
        ii. \& iii. These claims follow from the monotonicity of $j \mapsto p_j$ and Propositions \ref{prp: j=0} and \ref{prp: j_to_0}.
      \end{proof}
      
      In Fig. \ref{pic: pj}, we plot $p$ as a function of $j$ for $V(x)=\frac{1}{2}\sin(2\pi (x+\frac{1}{4}))$.
      
      \begin{proposition}
        Let $x=0$ be the only point of maximum of $V$. Then,
        \begin{itemize}
                \item[i.] for every $p \in \Rr,$ there exists a unique number, $\Hh(p),$ for which  \eqref{main} has a regular solution;
                \item[ii.] for every $p\in \Rr, $  \eqref{main} has a unique regular solution;
                \item[iii.] if $\max\limits_{\Tt}
                V> 1+\int\limits_{\Tt}V(x)dx$, $\Hh(p)$ is flat at the origin;
                \item[iv.] $\Hh(p)$ is increasing on $(0,\infty)$ and decreasing on $(-\infty,0)$. Thus 
                \[
                \min\limits_{p\in \Rr} \Hh(p)=\Hh(0)=\max \left(\max\limits_{\Tt} V, 1+\int\limits_{\Tt}V(x)dx\right);
                \]
                \item[v.] $\lim\limits_{|p| \to \infty} \frac{\Hh(p)}{p^2/2}=1$.
        \end{itemize}
      \end{proposition}
      \begin{proof}
        i \& ii. From Lemma \ref{lma: p-j}, we have that for every $p,$ there exists a unique $j$ such that \eqref{main} has regular solutions. From Proposition \ref{prp: H_j}, we have that for every $j$ there exists a unique number, $\Hh$, such that \eqref{main} has a regular solution. Therefore, for every $p,$ the constant $\Hh$ is determined uniquely. Moreover, from Proposition \ref{prp: j>0}, we have that \eqref{main} has a unique regular solution for this constant.
        
        iii. From iii. in Lemma \ref{lma: p-j}, we have that if $p$ satisfies \eqref{eq: pbound}, then $\Hh(p)=\Hh_0$.
        
        iv. This follows from i in Lemma \ref{lma: p-j} and iv and  v in Proposition \ref{prp: H_j}.
        
        v. This follows from vi in Proposition \ref{prp: H_j}, ii in Proposition \ref{prp: j_to_infty}, and the formula $p_j=\int\limits_{\Tt}\frac{j}{m_j(y)}dy$.
        
      \end{proof}
      In Fig. \ref{pic: Hp}, we show $\Hh(p)$ for $V(x)=\frac{1}{2}\sin(2\pi(x+\frac{1}{4}))$.

      \begin{figure}
        \includegraphics[width=50mm]{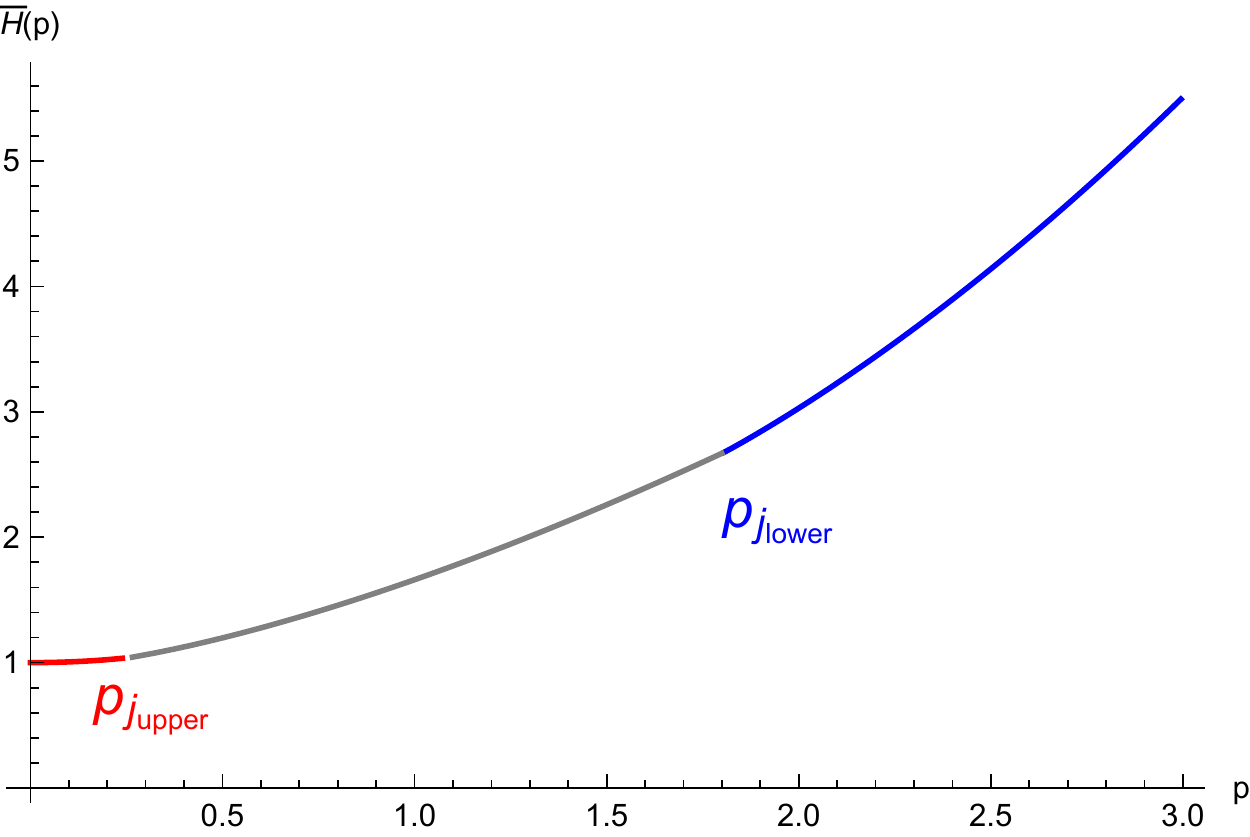}
        \caption{$\bar{H}(p)$ for $V(x)=\frac 1 2 \sin(2 \pi (x + \frac 1 4))$.}
        \label{pic: Hp}
      \end{figure}

      \begin{remark}
        We conjecture that, if $V$ has only one maximum point, $\Hh(p)$ is convex. Let $p_1<p_2$ and $(u_1,m_1,\Hh(p_1))$ and $\ (u_2,m_2,\Hh(p_2))$ solve \eqref{main} for $p=p_1$ and $p=p_2$, respectively. Consider the trajectories, $y_1,y_2,$ determined by        \begin{equation*}
                \dot{y}_i(t)=-(u_i)_x(y_i(t))-p_i,\quad y_i(0)=0,\quad i=1,2.
        \end{equation*}
        As in  weak Kolmogorov-Arnold-Moser (KAM) or classical KAM theory, we would like to show that
        \begin{equation}\label{eq: asymp_traj}
                \lim\limits_{t \to \infty} \frac{y_i(t)}{t}=\lim\limits_{t\to \infty} \dot{y}_i(t)=-D_p\Hh(p_i),\quad i=1,2.
        \end{equation}
        Furthermore, let $j_1$ and $j_2$ be the current values corresponding to $p_1$ and $p_2$. From the proof of i in Lemma \ref{lma: p-j}, we have that
        \begin{equation*}
                -(u_1)_x(y)-p_1=-\frac{j_1}{m_1(y)} \geq -\frac{j_2}{m_2(y)}=-(u_2)_x(y)-p_2.
        \end{equation*}
        Hence, if \eqref{eq: asymp_traj} holds, we get
        \[y_1(t)\geq y_2(t),\quad \dot{y}_1(t)\geq \dot{y}_2(t),\quad t>0,
        \]
        which implies $D_p\Hh(p_1) \leq D_p\Hh(p_2)$.  Thus, $\Hh(p)$ is convex.
      \end{remark}
      \begin{remark}
        If $V$ has more than one maximum point and is not constant, there exists a $p\in \Rr$ such that \eqref{main} has regular solutions for more than one value of $\Hh$. By Proposition \ref{prp: H_j}, we know that $\Hh$ is determined by the current level $j$. If $V$ has at least two maxima and is not constant, there exist multiple values for $p$ and solutions of \eqref{main} corresponding to a single value $j$ (see Proposition \ref{prp: V_multimax} and Remark \ref{rmrk: V_multimax_j=0}). Consequently, there exists a value of $p$ corresponding to different values of $j$ and, hence, to different values of $\Hh$.
      \end{remark}

\bibliographystyle{plain}
\bibliography{mfg}

\def\polhk#1{\setbox0=\hbox{#1}{\ooalign{\hidewidth
  \lower1.5ex\hbox{`}\hidewidth\crcr\unhbox0}}} \def\cprime{$'$}
\begin{thebibliography}{10}

\bibitem{Barlesbook}
G.~Barles.
\newblock {\em Solutions de viscosit\'e des \'equations de
  {H}amilton-{J}acobi}, volume~17 of {\em Math\'ematiques \& Applications
  (Berlin) [Mathematics \& Applications]}.
\newblock Springer-Verlag, Paris, 1994.

\bibitem{Braides}
Andrea Braides.
\newblock {\em {$\Gamma$}-convergence for beginners}, volume~22 of {\em Oxford
  Lecture Series in Mathematics and its Applications}.
\newblock Oxford University Press, Oxford, 2002.

\bibitem{cardaliaguet}
P.~Cardaliaguet.
\newblock Notes on mean-field games.
\newblock 2011.

\bibitem{Cd1}
P.~Cardaliaguet.
\newblock Long time average of first order mean-field games and weak kam
  theory.
\newblock {\em Preprint}, 2013.

\bibitem{Cd2}
P.~Cardaliaguet.
\newblock Weak solutions for first order mean-field games with local coupling.
\newblock {\em Preprint}, 2013.

\bibitem{cgbt}
P.~Cardaliaguet, P.~Garber, A.~Porretta, and D.~Tonon.
\newblock Second order mean field games with degenerate diffusion and local
  coupling.
\newblock {\em Preprint}, 2014.

\bibitem{MR3358627}
P.~Cardaliaguet and P.~J. Graber.
\newblock Mean field games systems of first order.
\newblock {\em ESAIM Control Optim. Calc. Var.}, 21(3):690--722, 2015.

\bibitem{FG2}
R.~Ferreira and D.~Gomes.
\newblock Existence of weak solutions for stationary mean-field games through
  variational inequalities.
\newblock {\em Preprint}.

\bibitem{GMit}
D.~Gomes and H.~Mitake.
\newblock Existence for stationary mean-field games with congestion and
  quadratic {H}amiltonians.
\newblock {\em NoDEA Nonlinear Differential Equations Appl.}, 22(6):1897--1910,
  2015.

\bibitem{GNP2}
D.~Gomes, L.~Nurbekyan, and M.~Prazeres.
\newblock Explicit solutions of one-dimensional, first-order, stationary
  mean-field games with congestion.
\newblock {\em Preprint}, 2016.

\bibitem{GPat2}
D.~Gomes and S.~Patrizi.
\newblock Weakly coupled mean-field game systems.
\newblock {\em Preprint}.

\bibitem{GPat}
D.~Gomes and S.~Patrizi.
\newblock Obstacle mean-field game problem.
\newblock {\em Interfaces Free Bound.}, 17(1):55--68, 2015.

\bibitem{GPatVrt}
D.~Gomes, S.~Patrizi, and V.~Voskanyan.
\newblock On the existence of classical solutions for stationary extended mean
  field games.
\newblock {\em Nonlinear Anal.}, 99:49--79, 2014.

\bibitem{GPim2}
D.~Gomes and E.~Pimentel.
\newblock Time dependent mean-field games with logarithmic nonlinearities.
\newblock {\em To appear in SIAM Journal on Mathematical Analysis}.

\bibitem{GPim1}
D.~Gomes and E.~Pimentel.
\newblock Local regularity for mean-field games in the whole space.
\newblock {\em To appear in Minimax Theory and its Applications}, 2015.

\bibitem{GP13}
D.~Gomes and E.~Pimentel.
\newblock Regularity for mean-field games with initial-initial boundary
  conditions.
\newblock In J.~P. Bourguignon, R.~Jeltsch, A.~Pinto, and M.~Viana, editors,
  {\em Dynamics, Games and Science III, CIM-MS}. Springer, 2015.

\bibitem{GPV}
D.~Gomes, E.~Pimentel, and V.~Voskanyan.
\newblock {\em Regularity theory for mean-field game systems}.
\newblock 2016.

\bibitem{GPM1}
D.~Gomes, G.~E. Pires, and H.~S{\'a}nchez-Morgado.
\newblock A-priori estimates for stationary mean-field games.
\newblock {\em Netw. Heterog. Media}, 7(2):303--314, 2012.

\bibitem{GR}
D.~Gomes and R.~Ribeiro.
\newblock Mean field games with logistic population dynamics.
\newblock {\em 52nd IEEE Conference on Decision and Control (Florence, December
  2013)}, 2013.

\bibitem{GM}
D.~Gomes and H.~S{\'a}nchez~Morgado.
\newblock A stochastic {E}vans-{A}ronsson problem.
\newblock {\em Trans. Amer. Math. Soc.}, 366(2):903--929, 2014.

\bibitem{GS}
D.~Gomes and J.~Sa{\'u}de.
\newblock Mean field games models---a brief survey.
\newblock {\em Dyn. Games Appl.}, 4(2):110--154, 2014.

\bibitem{GVrt2}
D.~Gomes and V.~Voskanyan.
\newblock Short-time existence of solutions for mean-field games with
  congestion.
\newblock {\em To appear J. London Math. Soc.}

\bibitem{GP2}
D.~A. Gomes and E.~Pimentel.
\newblock Regularity for mean-field games systems with initial-initial boundary
  conditions: the subquadratic case.
\newblock {\em To appear in Dynamics Games and Science III}, 2014.

\bibitem{GVrt}
Diogo~A. Gomes and Vardan~K. Voskanyan.
\newblock Extended deterministic mean-field games.
\newblock {\em SIAM J. Control Optim.}, 54(2):1030--1055, 2016.

\bibitem{Graber2}
J.~Graber.
\newblock Weak solutions for mean field games with congestion.
\newblock {\em Preprint}, 2015.

\bibitem{gueant3}
O.~Gu{\'e}ant.
\newblock A reference case for mean field games models.
\newblock {\em J. Math. Pures Appl. (9)}, 92(3):276--294, 2009.

\bibitem{Caines2}
M.~Huang, P.~E. Caines, and R.~P. Malham{\'e}.
\newblock Large-population cost-coupled {LQG} problems with nonuniform agents:
  individual-mass behavior and decentralized {$\epsilon$}-{N}ash equilibria.
\newblock {\em IEEE Trans. Automat. Control}, 52(9):1560--1571, 2007.

\bibitem{Caines1}
M.~Huang, R.~P. Malham{\'e}, and P.~E. Caines.
\newblock Large population stochastic dynamic games: closed-loop
  {M}c{K}ean-{V}lasov systems and the {N}ash certainty equivalence principle.
\newblock {\em Commun. Inf. Syst.}, 6(3):221--251, 2006.

\bibitem{ll1}
J.-M. Lasry and P.-L. Lions.
\newblock Jeux \`a champ moyen. {I}. {L}e cas stationnaire.
\newblock {\em C. R. Math. Acad. Sci. Paris}, 343(9):619--625, 2006.

\bibitem{ll2}
J.-M. Lasry and P.-L. Lions.
\newblock Jeux \`a champ moyen. {II}. {H}orizon fini et contr\^ole optimal.
\newblock {\em C. R. Math. Acad. Sci. Paris}, 343(10):679--684, 2006.

\bibitem{ll3}
J.-M. Lasry and P.-L. Lions.
\newblock Mean field games.
\newblock {\em Jpn. J. Math.}, 2(1):229--260, 2007.

\bibitem{llg2}
J.-M. Lasry, P.-L. Lions, and O.~Gu{\'e}ant.
\newblock Mean field games and applications.
\newblock {\em Paris-Princeton lectures on Mathematical Finance}, 2010.

\bibitem{PV15}
E.~Pimentel and V.~Voskanyan.
\newblock Regularity for second-order stationary mean-field games.
\newblock {\em To appear in Indiana University Mathematics Journal}.

\bibitem{porretta}
A.~Porretta.
\newblock On the planning problem for the mean field games system.
\newblock {\em Dyn. Games Appl.}, 4(2):231--256, 2014.

\bibitem{porretta2}
A.~Porretta.
\newblock Weak solutions to {F}okker-{P}lanck equations and mean field games.
\newblock {\em Arch. Ration. Mech. Anal.}, 216(1):1--62, 2015.

\end{thebibliography}

\end{document}